\documentclass[11pt]{amsart}
\usepackage{amsmath}
\usepackage{amssymb}
\usepackage{amsfonts}
\usepackage{amsthm}
\usepackage{enumerate}
\usepackage[mathscr]{eucal}
\usepackage{verbatim}
\usepackage{amsthm}
\usepackage{amscd}
\usepackage[mathscr]{eucal}
\usepackage{appendix}
\usepackage{tikz}
\usepackage{hyperref}
\usepackage{cleveref}

\numberwithin{equation}{section}

\newtheorem{innercustomgeneric}{\customgenericname}
\providecommand{\customgenericname}{}

\newcommand{\newcustomtheorem}[2]{\newenvironment{#1}[1]
  {\renewcommand\customgenericname{#2}
   \renewcommand\theinnercustomgeneric{##1}\innercustomgeneric}{\endinnercustomgeneric}}

\newcustomtheorem{customthm}{Theorem}

\newcommand{\newcustomlemma}[2]{\newenvironment{#1}[1]
  {\renewcommand\customgenericname{#2}
   \renewcommand\theinnercustomgeneric{##1} \innercustomgeneric}{\endinnercustomgeneric}}

\newcustomlemma{customlemma}{Lemma}

\newcustomlemma{customproposition}{Proposition}

\newcustomlemma{customclaim}{Claim}

\newcommand{\gap}{7}
\newcommand{\gaptt}{6}
\newcommand{\gapt}{7}

\theoremstyle{plain}
\newtheorem{theorem}{Theorem}

\newtheorem{proposition}[theorem]{Proposition}

\newtheorem*{theorem*}{Theorem}
\newtheorem*{lemma*}{Lemma}
\newtheorem*{proposition*}{Proposition}
\newtheorem*{corollary*}{Corollary}
\newtheorem*{remark*}{Remark} 
\newtheorem*{remarks*}{Remarks}
\newtheorem*{conj*}{Conjecture}

\def\f{\frac}

\def\Z{{\mathbb Z}}

\newcommand{\LL}{\mathcal{L}}
\newcommand{\HH}{\mathcal{H}}
\newcommand{\OO}{\mathcal{O}}

\newcommand{\bbz}{\mathbb{Z}}

\newcommand{\bbs}{\mathbb S}

\newcommand{\bbr}{\mathbb{R}}
\newcommand{\bbrn}{\mathbb R^n}

\newcommand{\bbn}{\mathbb{N}}

\newcommand{\xxxi}{\vec{\boldsymbol{\xi}\;}}

\newcommand{\yyy}{\vec{\boldsymbol{y}}}

\newcommand{\xxi}{\vec{\boldsymbol{\xi}}}

\def\000{\vec{\boldsymbol{0}}}

\def\ii{{\mathrm{i}}}

\def\ga{\gamma}

\def\j{\xi}

\def\ep{\epsilon}
\def\Om{\Omega}

\newcommand{\q}{\quad}
\newcommand{\qq}{\qquad}

\DeclareFontFamily{U}{mathx}{\hyphenchar\font45}
\DeclareFontShape{U}{mathx}{m}{n}{
	<5> <6> <7> <8> <9> <10>
	<10.95> <12> <14.4> <17.28> <20.74> <24.88>
	mathx10
}{}

\def\f{\frac}
\def\wh{\widehat}

\newcommand{\wt}{\widetilde}

\newcommand{\supp}{\mathrm{supp}}

\makeatletter
\@namedef{subjclassname@2020}{\textup{2020} Mathematics Subject Classification}
\makeatother

\allowdisplaybreaks

\makeatletter
\@namedef{subjclassname@2020}{\textup{2020} Mathematics Subject Classification}
\makeatother

\makeindex         

\begin{document}

\author{Bae Jun Park}
\address{B. Park, Department of Mathematics, Sungkyunkwan University, Suwon 16419, Republic of Korea}
\email{bpark43@skku.edu}

\thanks{The author is supported in part by NRF grant 2022R1F1A1063637 and by POSCO Science Fellowship of POSCO TJ Park Foundation. The author is grateful for support by the Open KIAS Center at Korea Institute
for Advanced Study.}

 \title{Multilinear estimates for maximal rough singular integrals} 
\subjclass[2020]{Primary 42B20, 42B25, 47H60}
\keywords{Multilinear estimates, Rough Singular integral operator, Maximal operators, Pointwise convergence}

\begin{abstract} 
In this work, we establish  $L^{p_1}\times \cdots\times L^{p_m}\to L^p$ bounds for maximal multi-(sub)linear singular integrals associated with homogeneous kernels $\frac{\Omega(\yyy')}{|\yyy|^{mn}}$
 where $\Omega$ is an $L^q$ function on the unit sphere $\mathbb{S}^{mn-1}$ with vanishing moment condition and $q>1$.
 As an application, we obtain almost everywhere convergence results for the associated doubly truncated multilinear singular integrals.
\end{abstract}

\maketitle


\section{Introduction}\label{introsec}

Let $n,m$ be integers with $n\ge 1$ and $m\ge 2$, and  consider an integrable function $\Omega$ on the unit sphere $\mathbb{S}^{mn-1}$ with the mean value zero property 
\begin{equation}\label{vanishingmtcondition}
\int_{\mathbb{S}^{mn-1}}\Omega(\yyy')~ d\sigma(\yyy')=0
\end{equation} 
where $d\sigma$ stands for the surface measure on $\mathbb{S}^{mn-1}$, $\yyy:=(y_1,\dots,y_m)\in (\bbrn)^m$, and $\yyy':=\frac{\yyy}{|\yyy|}\in \mathbb{S}^{mn-1}$.
We set
\begin{equation}\label{kernelrep}
K(\yyy):=\frac{\Omega(\yyy')}{|\yyy|^{mn}}, \qquad \yyy \neq \000, 
\end{equation} 
and define the corresponding truncated multilinear operator $\LL_{\Om}^{(\epsilon)}$  by
$$
\mathcal L^{(\ep)}_{\Om}\big(f_1,\dots,f_m\big)(x):=\int_{|\yyy|>\epsilon}{K(\yyy)\prod_{j=1}^{m}f_j(x-y_j)}~d\yyy
$$
 for Schwartz functions $f_1,\dots,f_m$ on $\bbrn$.
 By taking $\epsilon\searrow 0$, we also define the multilinear homogeneous singular integral operator 
\begin{align*}
\mathcal L_{\Om}\big(f_1,\dots,f_m\big)(x)&:=\lim_{\epsilon\searrow 0}\mathcal L^{(\ep)}_{\Om}\big(f_1,\dots,f_m\big)(x) =p.v. \int_{(\bbrn)^m}{K(\yyy)\prod_{j=1}^{m}f_j(x-y_j)}~d\yyy.
\end{align*}
This is still well-defined for any Schwartz functions $f_1, \dots,f_m$ on $\bbrn$.

There were several remarkable boundedness results in the linear setting ($m=1$ and $n\ge 2$) and  these results have been later extended to multilinear  cases when $m\ge 2$.
In this paper, we will mainly focus on the multilinear operator, leaving only some references \cite{Ca_Zy1956, Ch1988,Ch_Ru1988, Co_We1977, Co1979, Ho1988, Se1996, St2001, Ta1999} for the linear case, as many other relevant papers provide detailed historical background on the results for linear operators. 

The bilinear ($m=2$) singular integral operators in the one-dimensional setting $n=1$ were first studied by Coifman and Meyer in \cite{Co_Me1975} who established the $L^{p_1}(\bbr)\times L^{p_2}(\bbr)\to L^p(\bbr)$ boundedness for the bilinear operator $\LL_{\Omega}$ when $\Omega$ is a function of bounded variation on the unit circle $\mathbb{S}^1$, and this result was later extended to general dimensions $n\ge 1$ and $m$-linear operators ($m\ge 2$) by Grafakos and Torres \cite{Gr_To2002} who assumed $\Omega$ is a Lipschitz function on $\mathbb{S}^{mn-1}$. Both results need some smoothness assumptions on $\Omega$ and the results  were developed in the bilinear case by Grafakos, He, and Honz\'ik \cite{Gr_He_Ho2018} who  addressed the case when $\Omega$ merely belongs to $L^{\infty}(\mathbb{S}^{2n-1})$. Especially, they obtained the initial estimate $L^2\times L^2\to L^1$ for $\LL_{\Omega}$ even when $\Omega\in L^2(\mathbb{S}^{2n-1})$, introducing a new approach using a wavelet decomposition of Daubechies in \cite{Da1988}. The initial estimate was soon improved by Grafakos, He, and Slav\'ikov\'a \cite{Gr_He_Sl2020} who weakened the assumption $\Omega\in L^2(\mathbb{S}^{2n-1})$ to $\Omega\in L^q(\mathbb{S}^{2n-1})$ for $q>\frac{4}{3}$, and this result was extented to arbitrary exponent $1<p_1,p_2<\infty$ and $\frac{1}{2}<p<\infty$ by He and the author in \cite{He_Park2023} under the assumption that $\Omega\in L^q(\mathbb{S}^{2n-1})$ for $q>\max{(\frac{4}{3},\frac{p}{2p-1})}$.
 For general multilinear cases, Grafakos, He, Honz\'ik, and the author \cite{Gr_He_Ho_Park2023} derived an initial boundedness result 
 $L^2\times \cdots\times L^2\to L^{\frac{2}{m}}$ when $\Omega\in L^{q}(\mathbb{S}^{mn-1})$ for $q>\frac{2m}{m+1}$. The wavelet decomposition of Daubechies was still an essential tool in the multilinear case, but more intricate technical issues emerged  as the target space $L^{\frac{2}{m}}(\bbrn)$ is not a Banach space when $m\ge 3$. Later, the multilinear initial estimate was generalized to the whole range $1<p_1,\dots,p_m<\infty$ and $\frac{1}{m}<p<\infty$ in \cite{Gr_He_Ho_Park_JLMS}, and Dosidis and Slav\'ikov\'a \cite{Do_Sl2024} improved the estimates in a certain range of $p_1,\dots,p_m$. Interestingly, they proved that $\Omega\in L^q(\mathbb{S}^{mn-1})$ for $q>1$ is enough for the $L^{p_1}\times \cdots\times L^{p_m}\to L^p$ boundedness to hold when $1<p,p_1,\dots,p_m<\infty$.

In order to comprehensively describe all of the above results, let us introduce some notation.
Let $J_m:=\{1,\dots,m\}$.
For $0\le s\le 1$ and any subsets $J\subseteq  J_m$, let
\begin{equation*}
\HH_J^m(s):=\Big\{(t_1,\dots,t_m)\in (0,1)^m: \sum_{j\in J}(s-t_j)>-(1-s) \Big\},
\end{equation*}
\begin{equation*}
\OO_J^m(s):=\Big\{(t_1,\dots,t_m)\in (0,1)^m: \sum_{j\in J}(s-t_j)<-(1-s) \Big\}
\end{equation*}
and we define
\begin{equation*}
\HH^m(s):=\bigcap_{J\subseteq J_m}\HH_J^m(s).
\end{equation*}
See Figure \ref{fighs} for the shape of  $\mathcal{H}^3(s)$ in the trilinear case.
We observe that 
$$\mathcal{H}^m(s_1)\subset \mathcal{H}^m(s_2)\subset (0,1)^m \q \text{ for }~s_1<s_2$$
and $\lim_{s\nearrow 1}\mathcal{H}^m(s)=\mathcal{H}^m(1)=(0,1)^m$.
Moreover, 
\begin{align*}
\HH^m(0)&=\Big\{(t_1,\dots,t_m)\in (0,1)^m: t_1+\cdots+t_m<1 \Big\}.
\end{align*}
We also define the rectangle 
\begin{equation}\label{defvlm}
\mathscr{V}^m_l(s):=\{(t_1,\dots,t_m): 0< t_l< 1 \q \text{and }~ 0<  t_j< s ~\text{ for }~ j\not= l\}
\end{equation}
for $l\in J_m$ and $s>0$.
As known in \cite[Lemma 5.4]{Gr_He_Ho_Park_JLMS}, if $0<s<1$, then
\begin{equation}\label{convexhull}
\text{$\mathcal H^m(s)$
is the convex hull of the  rectangles
$\mathscr{V}^m_l(s)$, $l=1,\dots,m$.}
\end{equation}

 \begin{figure}[h]
\begin{tikzpicture}

\path[fill=green!5] (0,0,3)--(1,0,3)--(2.25,0,1.25)--(2.25,0,0)--(2.25,0.85,0)--(1,2.25,0)--(0,2.25,0)--(0,2.25,1.25)--(0,0.85,3)--(0,0,3);

\draw[dash pattern= { on 2pt off 1pt}](1,0.85,3)--(1,2.25,1.25)--(2.25,0.85,1.25)--(1,0.85,3);
\draw[dash pattern= { on 2pt off 1pt}] (1,0.85,3)--(1,0,3)--(2.25,0,1.25)--(2.25,0.85,1.25);
\draw[dash pattern= { on 2pt off 1pt}] (1,2.25,1.25)--(1,2.25,0)--(2.25,0.85,0)--(2.25,0.85,1.25);
\draw[dash pattern= { on 2pt off 1pt}] (1,0.85,3)--(0,0.85,3)--(0,2.25,1.25)--(1,2.25,1.25);
\draw[dash pattern= { on 2pt off 1pt}] (2.25,0,1.25)--(2.25,0,0)--(2.25,0.85,0);
\draw[dash pattern= { on 2pt off 1pt}] (0,2.25,1.25)--(0,2.25,0)--(1,2.25,0);
\draw[dash pattern= { on 2pt off 1pt}] (1,0,3)--(0,0,3)--(0,0.85,3);

\draw[dash pattern= { on 1pt off 1pt}] (0,0,0)--(2.25,0,0);
\draw[dash pattern= { on 1pt off 1pt}] (0,0,0)--(0,2.25,0);
\draw[dash pattern= { on 1pt off 1pt}] (0,0,0)--(0,0,3);
\draw [->] (0,0,3)--(0,0,4);
\draw [->] (0,2.25,0)--(0,2.85,0);
\draw [->] (2.25,0,0)--(3,0,0);

\node [below] at (3,0,0) {\tiny$t_1$};
\node [left] at (0,2.85,0) {\tiny$t_2$};
\node [below] at (0,0,4) {\tiny$t_3$};

\node[above right] at (2.25,0,0) {\tiny$(1,0,0)$};

\node[left] at (0,2.25,0) {\tiny$(0,1,0)$};

\node[left] at (0,0,3) {\tiny$(0,0,1)$};

\node[right] at (2.25,0,1.25) {\tiny$(1,0,s)$};

\node[right] at (1.2,0.9,1.25) {\tiny$(1,s,s)$};

\node[below] at (1,0,3) {\tiny$(s,0,1)$};

\node[right] at (0.9,0.75,3) {\tiny$(s,s,1)$};

\node[left] at (0,0.85,3) {\tiny$(0,1,s)$};

\node[left] at (0,2.25,1.25) {\tiny$(0,s,1)$};

\node[above] at (1,2.25,0) {\tiny$(s,1,0)$};

\node[right] at (0.9,2.25,1.95) {\tiny$(s,1,s)$};

\node[right] at (2.25,0.85,0) {\tiny$(1,s,0)$};

\end{tikzpicture}
\caption{The region $\mathcal{H}^3(s)$ }\label{fighs}
\end{figure}
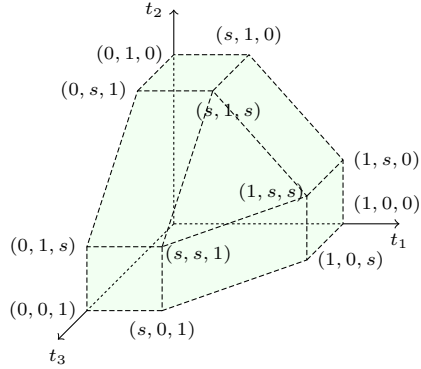

\begin{customthm}{A}\cite{Do_Sl2024, Gr_He_Ho2018, Gr_He_Ho_Park2023, Gr_He_Sl2020, He_Park2023}\label{knownbdresult}
Let $0<s<1$, $1<p_1,\dots,p_m<\infty$, and $\frac{1}{m}<p<\infty$ with $\frac{1}{p}=\frac{1}{p_1}+\cdots+\frac{1}{p_m}$.
Suppose that
$$\Big(\frac{1}{p_1},\cdots, \frac{1}{p_m} \Big)\in \HH^m(s)$$
and $\Omega\in L^{\frac{1}{1-s}}(\mathbb{S}^{mn-1})$ with \eqref{vanishingmtcondition}. Then there exists a constant $C>0$ such that
\begin{equation*}
\big\Vert \LL_{\Omega}(f_1,\dots,f_m)\big\Vert_{L^p(\bbrn)}\le C\Vert \Omega\Vert_{L^{\frac{1}{1-s}}(\mathbb{S}^{mn-1})}\prod_{j=1}^{m}\Vert f_j\Vert_{L^{p_j}(\bbrn)}
\end{equation*}
for Schwartz functions $f_1,\dots,f_m$ on $\bbrn$.
\end{customthm}
Setting $1<q=\frac{1}{1-s}<\infty$, Theorem \ref{knownbdresult} is equivalent to the statement that
\begin{equation}\label{lomelpest}
\big\Vert \LL_{\Omega}(f_1,\dots,f_m)\big\Vert_{L^p(\bbrn)}\le C\Vert \Omega\Vert_{L^{q}(\mathbb{S}^{mn-1})}\prod_{j=1}^{m}\Vert f_j\Vert_{L^{p_j}(\bbrn)} 
\end{equation} holds,
provided that $1<p_1,\dots,p_m<\infty$ and $\frac{1}{m}<p<\infty$ satisfy
\begin{equation}\label{qcondequi}
\sum_{j\in J}\frac{1}{p_j}<\frac{|J|}{q'}+\frac{1}{q} \q \text{ for any subsets $J$ of $J_m$}.
\end{equation}
We should also remark that   the estimate \eqref{lomelpest} in the bilinear setting has been recently further improved by Dosidis, Slav\'ikov\'a, and the author \cite{Do_Park_Sl_submitted} weakening the $L^q$ assumption on $\Omega$ to the requirement that $\Omega$ belongs to the Orlicz space $L(\log{L})^{\alpha}$ for some $\alpha>0$ when $1<p,p_1,p_2<\infty$, or equivalently $(\frac{1}{p_1},\frac{1}{p_2})\in \HH^2(0)$.

\hfill

In this paper we are primarily concerned with maximal multi-(sub)linear operators associated to the singular integral operator $\LL_{\Omega}$, defined by
\begin{equation*}
\LL_{\Om}^{*}\big(f_1,\dots ,f_m\big)(x) := 
\sup_{\ep>0 } \big| \LL_{\Om}^{(\ep)}\big(f_1,\dots , f_m\big)(x)\big|, \qquad x\in\bbrn 
\end{equation*}
for Schwartz functions $f_1,\dots,f_m$ on $\bbrn$.
Employing the wavelet decomposition used in the proof of initial estimates for $\LL_{\Omega}$, the $L^2\times \cdots\times L^2\to L^{\frac{2}{m}}$ boundedness result was obtained by Buri\'ankov\'a and Honz\'ik \cite{Bu_Ho2019} for bilinear maximal operators and by Grafakos, He, Honz\'ik, and the author \cite{Gr_He_Ho_Park2024} for general multilinear ones.
\begin{customthm}{B}\cite{Bu_Ho2019, Gr_He_Ho_Park2024}\label{maximalinitial}
Suppose that $\Omega$ satisfies \eqref{vanishingmtcondition} and
 \begin{equation}\label{initialcononome}
\Omega\in L^q(\mathbb{S}^{mn-1})\q \text{ for }~ \frac{2m}{m+1}<q\le \infty.
\end{equation} 
Then there exists a constant $C>0$ such that
\begin{equation}\label{earlierinitialmulti}
\big\Vert \LL_{\Omega}^*(f_1,\dots,f_m)\big\Vert_{L^{\frac{2}{m}}(\bbrn)}\le C \Vert \Omega\Vert_{L^q(\mathbb{S}^{mn-1})}\prod_{j=1}^{m}\Vert f_j\Vert_{L^{2}(\bbrn)}
\end{equation}
for Schwartz functions $f_1,\dots,f_m$ on $\bbrn$.
\end{customthm}

The main result of this paper is the following general $L^{p_1}\times \cdots\times L^{p_m}\to L^p$ estimate for $\LL_{\Omega}^*$,
which extends and improves the initial estimate in Theorem \ref{maximalinitial} to all indices $1<p_1,\dots, p_m<\infty$ and $\frac{1}{m}<p<\infty$ under the same hypothesis on $\Omega$ as in Theorem \ref{knownbdresult}. 

\begin{theorem}\label{maintheorem}
Let $0<s<1$, $1<p_1,\dots,p_m<\infty$, and $\frac{1}{m}<p<\infty$ with $\frac{1}{p}=\frac{1}{p_1}+\cdots+\frac{1}{p_m}$.
Suppose that
\begin{equation}\label{rangeofpj}
\Big(\frac{1}{p_1},\dots,\frac{1}{p_m} \Big)\in\mathcal{H}^m(s)
\end{equation}
and $\Omega\in L^{\frac{1}{1-s}}(\mathbb{S}^{mn-1})$ with  \eqref{vanishingmtcondition}.
Then there exists a constant $C>0$ such that
\begin{equation*}
 \big\|\LL_{\Om}^{*}(f_1,\dots , f_m)\big\|_{L^{p}(\mathbb{R}^n) } \le C\Vert \Omega\Vert_{L^{\frac{1}{1-s}}(\mathbb{S}^{mn-1})} \prod_{j=1}^{m}\| f_j\|_{L^{p_j}(\mathbb{R}^n)}
 \end{equation*}
for Schwartz functions $f_1,\dots,f_m$ on $\bbrn$.
\end{theorem}
We point out that Theorem \ref{maintheorem} deduces that the same initial multilinear estimate \eqref{earlierinitialmulti} holds even for $\frac{2(m-1)}{m}<q\le \frac{2m}{m+1}$, which improves Theorem \ref{maximalinitial}.\\

As is generally known (even in the linear setting), such a maximal function estimate is related to a problem of almost everywhere convergence of the associated doubly truncated singular integrals
$$
\mathcal L^{(\ep,\ep^{-1})}_{\Om}\big(f_1,\dots,f_m\big)(x):=\int_{\epsilon<|\yyy|<\epsilon^{-1}}{K(\yyy)\prod_{j=1}^{m}f_j(x-y_j)}~d\yyy
$$
as $\epsilon \searrow 0$ in the case that each $f_j$ is an $L^{p_j}$ function on $\bbrn$. Indeed, it is proved in \cite[Theorem 1.1]{Gr_He_Ho_Park2024} that  
\begin{equation}\label{ptaeconinitial}
\LL^{(\ep,\ep^{-1})}_{\Om}\big(f_1,\dots,f_m\big)(x)\to \LL_{\Omega}(f_1,\dots,f_m)(x) \q \text{ a.e. \q  as \q $\epsilon \searrow 0$}
\end{equation}
when  $f_1,\dots,f_m\in L^2(\bbrn)$ and $\Omega\in L^q(\mathbb{S}^{mn-1})$ for $\frac{2m}{m+1}<q\le \infty$, applying Theorem \ref{maximalinitial}. 
Similarly, as an application of Theorem \ref{maintheorem}, we obtain the following almost everywhere pointwise estimate.
\begin{theorem}\label{maincor}
Let $1<p_1,\dots,p_m<\infty$ and $1<q\le \infty$ with \eqref{qcondequi}.
Suppose that $\Omega\in L^q(\mathbb{S}^{mn-1})$ satisfies \eqref{vanishingmtcondition}.
Then for each $f_j\in L^{p_j}(\bbrn)$, the doubly truncated singular integral
$\LL_{\Omega}^{(\epsilon,\epsilon^{-1})}(f_1,\dots,f_m)$ converges to $\LL_{\Omega}(f_1,\dots,f_m)$ pointwise almost everywhere as $\epsilon\searrow 0$.
\end{theorem}

As a consequence of Theorem \ref{maincor}, the multilinear singular integral $\LL_{\Omega}(f_1,\dots,f_m)$ is well-defined almost everywhere when $f_j\in L^{p_j}(\bbrn)$, $j=1,\dots,m$.
Theorem \ref{maincor} can be proved by replacing Theorem \ref{maximalinitial} with Theorem \ref{maintheorem} and then simply mimicking the proof of \eqref{ptaeconinitial} in \cite{Gr_He_Ho_Park2024}. For the sake of completeness, we include the proof in the appendix.

\hfill

In order to prove Theorem \ref{maintheorem}, we apply a dyadic decomposition introduced by Duoandikoetxea and Rubio de Francia \cite{Du_Ru1986}, which has already been employed very essentially in many earlier papers \cite{Bu_Ho2019, Do_Park_Sl_submitted, Do_Sl2024, Gr_He_Ho2018, Gr_He_Ho_Park2023, Gr_He_Ho_Park2024, Gr_He_Ho_Park_JLMS, Gr_He_Sl2020, He_Park2023},
 and utilize the same reduction step as in the proof of Theorem \ref{maximalinitial} in \cite{Gr_He_Ho_Park2024}.
More precisely, we decompose the kernel $K$ in \eqref{kernelrep} as 
$$K=\sum_{\mu\in\bbz}\sum_{\gamma\in \bbz}K_{\mu}^{\gamma}$$
where $K_{\mu}^{\gamma}=\Psi_{\mu+\gamma}\ast \big(K \cdot\wh{\Psi_{-\gamma}}\big)$ and $\Psi_k$ is a Littlewood-Paley function on $(\bbrn)^m$, which will be officially defined in Section \ref{preliminarysection}, whose Fourier transform is supported in an annulus of size $2^{k}$. Then the maximal function $\LL_{\Omega}^*(f_1,\dots,f_m)$ can be estimated as
$$\LL_{\Omega}^*(f_1,\dots,f_m)\le \mathcal{M}_{\Omega}(f_1,\dots,f_m)+\LL_{\Omega}^{\sharp}(f_1,\dots,f_m)$$
where 
\begin{equation*}
\mathcal M_\Om\big(f_1,\dots, f_m\big)(x)=
\sup_{R>0}  \f{1}{R^{mn} } \int_{|\yyy|\le R}   |\Om (\yyy' ) | 
\prod_{j=1}^m \big|  f_j(x-y_j) \big|  ~ d\yyy
\end{equation*}
and
\begin{equation}\label{e000sharp}
\LL_{\Om}^{\sharp}\big(f_1,\dots,f_m\big)(x) := 
\sup_{\tau\in \mathbb Z} \Big| \sum_{\gamma<\tau} \sum_{\mu\in\bbz}
T_{K_{\mu}^{\gamma}}\big(f_1,\dots,f_m\big)(x)\Big|.
\end{equation}
A boundedness result for $\mathcal{M}_{\Omega}$, which is required for the proof of  Theorem \ref{maintheorem}, has already been shown in \cite{Gr_He_Ho_Park2024}, and thus we only need to consider the remaining operator $\LL_{\Omega}^{\sharp}$.
We also notice that when the sum over $\mu\in\bbz$ in \eqref{e000sharp} changes to the sum over $\mu\le 0$, the corresponding operator satisfies the $L^{p_1}\times \cdots \times L^{p_m}\to L^p$ boundedness with a constant $C_q \Vert\Omega\Vert_{L^q(\mathbb{S}^{mn-1})}$ for any $1<q<\infty$ and $1<p_1,\dots,p_m\le \infty$ with $\frac{1}{p}=\frac{1}{p_1}+\cdots+\frac{1}{p_m}$. This was verified in \cite[Proposition 4.1]{Gr_He_Ho_Park2024}, using  multilinear version of Cotlar's inequality in \cite[Theorem 1]{GT-Indiana}, together with the fact that
$\sum_{\gamma\in\bbz}\sum_{\mu\le 0}K_{\mu}^{\gamma}$
is an $m$-linear Calder\'on-Zygmund kernel with constant $C_q\Vert\Omega\Vert_{L^q(\mathbb{S}^{mn-1})}$, thanks to the estimate of Duoandikoetxea and Rubio de Francia \cite{Du_Ru1986}; see \eqref{kmuhatest} below. Therefore, it suffices to deal with the case $\mu> 0$ in \eqref{e000sharp}, which is clearly bounded by
$$\sum_{\mu> 0}\LL_{\Omega,\mu}^{\sharp}(f_1,\dots,f_m)$$
where
$$
  \LL_{\Om,\mu}^{\sharp}\big(f_1,\dots,f_m\big)(x) 
  := \sup_{\tau\in \mathbb Z} \Big|   \sum_{\gamma<\tau} T_{K_{\mu}^{\gamma}}\big(f_1,\dots,f_m\big)  (x)  \Big| .
 $$
We will actually prove that there exists $\epsilon_0>0$ such that
 \begin{equation}\label{mainkeyest22_1}
\big\|  \LL_{\Om,\mu}^{\sharp}\big(f_1,\dots,f_m\big) \big\|_{L^{p}(\bbrn) } \lesssim_{\epsilon_0}  2^{-\epsilon_0 \mu}\| \Om\|_{L^{\frac{1}{1-s}}(\mathbb{S}^{mn-1})}\prod_{j=1}^{m}\Vert f_j\Vert_{L^{p_j}(\bbrn)}, \q \mu> 0
\end{equation}
when \eqref{rangeofpj} holds.
We remark that the structure of the proof is almost same as that of Theorem \ref{maximalinitial} in \cite{Gr_He_Ho_Park2024}
where one of the key estimates is
 \begin{equation}\label{mainkeyestinitial_1}
\big\|  \LL_{\Om,\mu}^{\sharp}\big(f_1,\dots,f_m\big) \big\|_{L^{\frac{2}{m}}(\bbrn) } \lesssim  2^{-\delta_0 \mu}\| \Om\|_{L^{q}(\mathbb{S}^{mn-1})}\prod_{j=1}^{m}\Vert f_j\Vert_{L^{2}(\bbrn)}, \q \mu> 0
\end{equation}
for  some $\delta_0>0$, provided that $q>\frac{2m}{m+1}$.
Therefore the main contribution of this work is to improve and extend the estimate \eqref{mainkeyestinitial_1} to \eqref{mainkeyest22_1}.
This will be achieved by establishing Propositions \ref{keyproposharp} and \ref{hsmainestpropo} in which analogous (but a slightly weaker) multilinear estimates are provided with arbitrary slow exponential growths in $\mu$, but will be finally improved to \eqref{mainkeyest22_1} by applying a decomposition of $\Omega$ based on its size; see \eqref{omegadecomp} below. 
It should be also mentioned that  we follow the terminology in \cite{Gr_He_Ho_Park2024} for the sake of unity as some of the results verified there will be used in the proof of Theorem \ref{maintheorem}.

\hfill

{\bf Organization.}
Section \ref{preliminarysection} contains some preliminary materials including several maximal inequalities, shifted operators, multilinear paraproducts, and multi-sublinear interpolation theory.
We will prove Theorem \ref{maintheorem} in Section \ref{proofmainresult}, presenting two key propositions, namely Propositions \ref{keyproposharp} and \ref{hsmainestpropo}. The proof of the two propositions will be given in turn in the next two sections.\\

\medskip
\noindent {\bf Acknowledgment:} The author would like to thank the anonymous referees for 
 their careful reading and valuable comments, which made this paper more readable. 
The author also thanks Stefanos Lappas for his useful comments.

\section{Preliminaries} \label{preliminarysection}

\subsection{Maximal inequalities}\label{maximalines}
We first recall some fundamental maximal inequalities.
For a locally integrable function $f$ defined on $\bbrn$, let 
$$
\mathcal{M}f(x):=\sup_{Q:x\in Q}\frac{1}{|Q|}\int_Q{|f(y)|}dy
$$
 be the Hardy-Littlewood maximal function of $f$ where the supremum is taken over all cubes in $\bbrn$ containing $x$, and let $\mathcal{M}_rf(x):=\big( \mathcal{M}\big(|f|^r\big)(x) \big)^{\frac{1}{r}}$ for $0<r<\infty$. Then the maximal operator $\mathcal{M}_r$ is bounded in $L^p$ when $0<r<p$ and
Fefferman and Stein \cite{Fe_St1971}  obtained a vector-valued counterpart;  for $0<p<\infty$, $0<q\leq \infty$, and $0<r<\min{(p,q)}$ one has
\begin{equation}\label{hlmax}
\big\Vert \big\{\mathcal{M}_rf_k\big\}_{k\in\mathbb{Z}}\big\Vert_{L^p(\ell^q)}\lesssim \big\Vert \{ f_k\}_{k\in\mathbb{Z}}\big\Vert_{L^p(\ell^q)}.
\end{equation}
Clearly, (\ref{hlmax}) also holds when $p=q=\infty$.

Given $k\in \Z$ and $\sigma>0$, we also introduce Peetre's maximal function in \cite{Pe1975}
\begin{equation*}
\mathfrak{M}_{\sigma,2^k}f(x):=\sup_{y\in\bbrn}{\frac{|f(x-y)|}{(1+2^k|y|)^{\sigma}}}.
\end{equation*}
 For $A>0$, let $\mathcal{E}(A)$ denote the space of all distributions whose Fourier transform is supported in $\big\{\xi\in\bbrn:|\xi|\leq 2A\big\}$.
 It turned out that
 \begin{eqnarray}\label{maximalbound}
\mathfrak{M}_{\frac{n}{r},2^k}f(x)\lesssim_{r,A} \mathcal{M}_rf(x), 
\end{eqnarray}  provided that $f\in\mathcal{E}(A2^k)$ for $A>0$.
A combination of~\eqref{maximalbound} and \eqref{hlmax} yields that 
 for $0<p<\infty$ and $0<q\leq \infty$, we have 
\begin{equation}\label{peetremax}
\big\Vert  \big\{\mathfrak{M}_{\sigma,2^k}f_k \big\}_{k\in\bbz} \big\Vert_{L^p(\ell^q)} \lesssim_{A,p,q}  \big\Vert \big\{f_k\big\}_{k\in\bbz}    \big\Vert_{L^p(\ell^q)} \quad \text{for}~\sigma>\frac{n}{\min{(p,q)}}
\end{equation}  if $f_k\in\mathcal{E}(A2^k)$.
Clearly, the above inequality also holds for $p=q=\infty$.\\

\subsection{Shifted operators}
Let $\phi$ and $\psi$ stand for Schwartz functions on $\bbrn$ such that
$$\wh{\phi}(0)=1,\q \supp{(\wh{\phi})}\subset \{\xi\in\bbrn: |\xi|\lesssim 1\},$$
\begin{equation*}
\supp{(\wh{\psi})}\subset \{\xi\in\bbrn: |\xi|\sim 1\}, \q \text{ and }\q \sum_{k\in\bbz}\wh{\psi_k}(\xi)=1,~ \xi\not= 0
\end{equation*}
where we set $\phi_k:=2^{kn}\phi(2^k\cdot)$ and $\psi_k:=2^{kn}\psi(2^k\cdot)$ for $k\in\bbz$.
It is easy to verify that for each $k\in\bbz$
\begin{equation}\label{phikfptest}
\big|\phi_k\ast f(x)\big|, \big|\psi_k\ast f(x)\big| \lesssim \mathcal{M}f(x) \quad \text{uniformly in } ~k
\end{equation}
and for any $\sigma>0$
\begin{equation}\label{phikfpetreept}
\big|\phi_k\ast f(x)\big|, \big| \psi_k\ast f(x)\big| \lesssim_{\sigma} \mathfrak{M}_{\sigma,2^k}f(x) \quad \text{uniformly in } ~k.
\end{equation}
Then we have the following characterizations of the Lebesgue space;
\begin{equation}\label{lpequivs}
\Vert f\Vert_{L^p(\bbrn)}\sim  \Big\Vert \sup_{k\in\bbz} \big| \phi_k\ast f \big|\Big\Vert_{L^p(\bbrn)}\sim \bigg\Vert \bigg( \sum_{k\in\bbz}\big| \psi_k\ast f \big|^2\bigg)^{\frac{1}{2}}\bigg\Vert_{L^p(\bbrn)} \q \text{ for }~1<p<\infty.
\end{equation}
The first equivalence follows from the Lebesgue differentiation theorem and the $L^p$ boundedness of $\mathcal{M}$ together with \eqref{phikfptest}. The second one is known as Littlewood-Paley theory.
The second equivalence of \eqref{lpequivs}, the pointwise estimate \eqref{phikfptest}, and the maximal inequality \eqref{hlmax} deduce the following estimate,
 which is very useful to estimate sum over $k\in\bbz$ of functions with Fourier support in an annulus of size $2^k$.
If $1<p<\infty$ and  each $f_k\in \mathscr{S}'(\bbrn)$, $k\in\bbz$, satisfies
\begin{equation}\label{fkfrsup}
\supp(\wh{f_k})\subset \{\xi\in\bbrn: C^{-1}2^k\le |\xi|\le C 2^k\}
\end{equation}
for some $C>1$, then 
we have
\begin{equation}\label{ltchaest}
\bigg\Vert \sum_{k\in\bbz}f_k\bigg\Vert_{L^p(\bbrn)}\lesssim_C \bigg\Vert \Big( \sum_{k\in\bbz} \big| f_k\big|^2\Big)^{\frac{1}{2}}\bigg\Vert_{L^p(\bbrn)}.
\end{equation}
Indeed, the left-hand side is equivalent to
\begin{align*}
 \bigg\Vert \bigg(\sum_{l\in\bbz}\Big| \sum_{k\in\bbz}\psi_l\ast f_k \Big|^2 \bigg)^{\frac{1}{2}}\bigg\Vert_{L^p(\bbrn)}&=\bigg\Vert \bigg(\sum_{l\in\bbz}\Big| \sum_{k=-B}^{B}\psi_l\ast f_{k+l} \Big|^2 \bigg)^{\frac{1}{2}}\bigg\Vert_{L^p(\bbrn)}\\
 &\lesssim \sum_{k=-B}^{B}\bigg\Vert \bigg(\sum_{l\in\bbz}\big| \mathcal{M}f_{k+l}\big|^2\bigg)^{\frac{1}{2}}\bigg\Vert_{L^p(\bbrn)}\\
&\le \sum_{k=-B}^{B}\bigg\Vert \Big(\sum_{l\in\bbz}\big| f_{k+l}\big|^2\Big)^{\frac{1}{2}}\bigg\Vert_{L^p(\bbrn)}\sim_B \bigg\Vert \Big(\sum_{k\in\bbz}\big| f_{k}\big|^2\Big)^{\frac{1}{2}}\bigg\Vert_{L^p(\bbrn)}
\end{align*}
for some nonnegative integer $B$, depending on $C$ in \eqref{fkfrsup}.

For $k\in\bbz$ and $y\in\bbrn$, we now define two shifted operators
$$(\psi_k)^y:=\psi_k(\cdot-2^{-k}y)=2^{kn}\psi(2^k\cdot-y)$$
and
$$(\phi_k)^y:=\phi_k(\cdot-2^{-k}y)=2^{kn}\phi(2^k\cdot-y).$$
Then  one direction of the two equivalences \eqref{lpequivs} can be generalized as follows.
\begin{customlemma}{C}\cite[Theorem 1.5, Corollary 1.7]{Park_submitted}\label{shiftsquare}
Let $1<p<\infty$ and $y\in\bbrn$.
Then we have
$$\bigg\Vert \sup_{k\in\bbz}\big| (\phi_k)^y\ast f\big|\bigg\Vert_{L^p(\bbrn)}\lesssim \big( \ln{(e+|y|)}\big)^{\frac{1}{p}}\Vert f\Vert_{L^p(\bbrn)}$$
and
\begin{equation*}
\bigg\Vert \Big( \sum_{k\in\bbz}\big|(\psi_k)^y\ast f\big|^2 \Big)^{\frac{1}{2}}\bigg\Vert_{L^p(\bbrn)}\lesssim \big(\log{(e+|y|)}\big)^{|\frac{1}{p}-\frac{1}{2}|}\Vert f \Vert_{L^p(\bbrn)}
\end{equation*}
where the constants in the inequalities do not depend on $y$.
\end{customlemma}
Weaker versions of such inequalities appeared in \cite[Theorems 4.5, 4.6]{Mu_Sc2013} for one-dimensional case and in \cite[Proposition 7.5.1]{MFA} and \cite[Corollary 1]{Gr_Oh2014} for higer-dimensional ones.
A different proof of the shifted square function estimate is given in \cite{Do_Park_Sl_submitted} as well.\\

\subsection{Multilinear paraproducts}
We now consider a multilinear paraproduct,  which is required in the proof of  Proposition \ref{keyproposharp}.
Let $\Psi$ be a Schwartz function on $(\bbrn)^m$ whose Fourier transform is supported in the annulus $\{\xxxi\in (\bbrn)^m: \frac{1}{2}\le |\xxxi|\le 2\}$ and satisfies
$\sum_{k\in\bbz}\wh{\Psi_k}(\xxxi)=1$ for $\xxxi\not= \000$ where $\wh{\Psi_k}(\xxxi):=\wh{\Psi}(2^{-k} \xxxi )$.
\begin{customlemma}{D}\cite[Lemma 4.1]{LHHLPPY} \label{yyvt}	
The term
$$\sum_{k \in \mathbb{Z}} \sum_{k_1, k_2, \cdots, k_m \in \mathbb{Z}}  \widehat{\Psi_k}(\xxi) \widehat{\psi_{k_1}}(\xi_1)\widehat{\psi_{k_2}}(\xi_2) \cdots \widehat{\psi_{k_{m}}}(\xi_{m})$$ 
can be written as a finite sum of form
\[\sum_{k \in \mathbb{Z}}\widehat{\Psi_k}(\xxi) \widehat{\Phi^1_k}(\xi_1)\widehat{\Phi^2_k}(\xi_2)\cdots\widehat{\Phi^m_k}(\xi_m) \widehat{\Phi^{m+1}_k}(-\xi_1-\cdots-\xi_m),\]
where $\xxi=(\xi_1, \xi_2, \cdots, \xi_m)\in (\bbrn)^m$, and $\widehat{\Phi^1}$, $\widehat{\Phi^2}$, $\cdots$, $\widehat{\Phi^{m+1}}$ are compactly supported radial smooth functions and at least two of them are compactly supported away from the origin, and $\wh{\Phi}_{k}^{j}:=\wh{\Phi}^j(2^{-k}\cdot)$ for $1\le j\le m+1$.
\end{customlemma}
Such a decomposition has already been used very effectively in \cite{Do_Sl2024, Gr_Park2021, Gr_Park2022,  LHHLPPY, Lee_Park_accepted, Mu_Pi_Ta_Th2004, Mu_Sc2013, Park_IUMJ}, where it reduces various multilinear operator problems into simpler forms, performing an analogous role to the Littlewood-Paley decomposition technique in the linear case.\\

\subsection{Interpolation theory for multi-sublinear operators}

We end this section by presenting a multi-sublinear version of the Marcinkiewicz interpolation theorem, which is a straightforward corollary of \cite[Theorem 1.1]{Gr_Li_Lu_Zh2012}.
\begin{customlemma}{E}\cite{ Gr_Li_Lu_Zh2012}\label{interpollemma}
Let $0<p_{j}^{\ii}\le \infty$ for each $j\in \{1,\dots,m\}$ and  $\ii=0,1,\dots, m$, and $0<p^{\ii}\le \infty$ satisfy $\frac{1}{p^{\ii}}=\frac{1}{p^{\ii}_{1}}+\dots+\frac{1}{p^{\ii}_{m}}$ for $\ii=0,1,\dots, m$.
Suppose that $T$ is an $m$-sublinear operator having the mapping properties
\begin{equation*}
\big\Vert T(f_1,\dots,f_m) \big\Vert_{ L^{p^{\ii},\infty}(\bbrn)}\le M_{\ii}\prod_{j=1}^{m}\Vert f_j\Vert_{L^{p_{j}^{\ii}}(\bbrn)}, \quad \ii=0,1,\dots, m
\end{equation*} 
for Schwartz functions $f_1,\dots, f_m$ on $\bbrn$.
Given  $0<\theta_{\ii}<1$ with $\sum_{\ii=0}^m\theta_{\ii}=1$,  set 
\begin{equation*}
\frac 1{ p_j}= \sum_{\ii=0}^m \frac{ \theta_{\ii}}{p_j^\ii}  , \q j\in J_m\qquad \text{ and}\qquad
\frac 1{p}=\sum_{\ii=0}^m  \frac{ \theta_{\ii}}{p^\ii }   .
\end{equation*} 
Then  we have
\begin{equation*}
\big\Vert T(f_1,\dots,f_m)\big\Vert_{L^{p,\infty}(\bbrn)}\lesssim 
M_0^{ \theta_0}\cdots M_m^{\theta_m} \prod_{j=1}^{m}\Vert f_j\Vert_{L^{p_j}(\bbrn)}
\end{equation*}
for Schwartz functions $f_1,\dots, f_m$ on $\bbrn$.
Moreover, if the  points $(\frac{1}{p_1^\ii},\dots , \frac{1}{p_m^\ii})$, $0\le \ii\le  m$, form a non trivial open simplex in $\mathbb R^m$, then 
\begin{equation*}
\big\Vert T(f_1,\dots,f_m)\big\Vert_{L^{p}(\bbrn)}\lesssim M_0^{ \theta_0}\cdots M_m^{\theta_m}\prod_{j=1}^{m}\Vert f_j\Vert_{L^{p_j}(\bbrn)} .
\end{equation*} 
\end{customlemma}

\hfill

\section{Proof of Theorem \ref{maintheorem}}\label{proofmainresult}

Let $\Psi$ and $\Psi_k$ be the Schwartz functions on $(\bbrn)^m$, introduced in Section \ref{preliminarysection}.
For each $\gamma,\mu \in\bbz$,  we define
 $$ K^{\gamma}(\yyy):=\wh{\Psi}(2^{\gamma}\yyy)K(\yyy)~ \text{ and }~ K_{\mu}^{\gamma}(\yyy):=\Psi_{{\mu}+\gamma}\ast K^{\gamma}(\yyy), \quad \yyy\in (\bbrn)^m.$$
 Then $K^\gamma(\yyy)=2^{\ga mn} K^0(2^\gamma \yyy)$ and this deduces
 \begin{equation}\label{kernelcharacter}
K_{\mu}^{\gamma}(\yyy)=2^{\gamma mn}\big(\Psi_{{\mu}}\ast K^{0}\big)(2^\gamma \yyy)=2^{\gamma mn}K_{\mu}^{0}(2^{\gamma}\yyy),
\end{equation}
or equivalently,
$$
\wh{K^\gamma_\mu}(\xxxi)= \wh{\Psi}(2^{-(\mu+\ga)}\xxxi)\wh{K^0}(2^{-\ga}\xxxi)=\wh{K^0_\mu}(2^{-\ga} \xxxi).
$$
 The associated operator $T_{K_{\mu}^{\gamma}}$ is defined as
$$T_{K_{\mu}^{\gamma}}\big(f_1,\dots,f_m\big)(x):=\int_{(\bbrn)^m}K_{\mu}^{\gamma}(\yyy)\prod_{j=1}^{m}f_j(x-y_j)\; d\yyy$$
so that
$$\LL_{\Omega}\big(f_1,\dots,f_m\big)=\sum_{\mu\in\bbz}\sum_{\gamma\in\bbz} T_{K_{\mu}^{\gamma}}\big(f_1,\dots,f_m\big).$$
Duoandikoetxea and Rubio de Francia \cite{Du_Ru1986} proved that if $1<q<\infty$ and $0<\delta<\frac{1}{q'}$, then 
 \begin{equation}\label{khat0pakh0est}
 \begin{aligned}
 \big| \wh{K^0}(\xxxi)\big|&\lesssim \Vert \Omega\Vert_{L^q(\mathbb{S}^{mn-1})}\min\big\{|\xxxi|,|\xxxi|^{-\delta} \big\}\\
  \big| \partial^{\alpha}\wh{K^0}(\xxxi)\big|&\lesssim \Vert \Omega\Vert_{L^q(\mathbb{S}^{mn-1})}\min\big\{1,|\xxxi|^{-\delta} \big\}, \qq \alpha\not= \vec{0}
 \end{aligned}
 \end{equation}
 and accordingly,
  \begin{align*}
 \Big| \sum_{\gamma\in\bbz}\wh{K_{\mu}^{\gamma}}(\xxxi)\Big|&\lesssim \Vert \Omega\Vert_{L^q(\mathbb{S}^{mn-1})}\min\big\{2^{\mu},2^{-\delta \mu} \big\}\\
  \Big| \sum_{\gamma\in\bbz}\partial^{\alpha}\wh{K_{\mu}^{\gamma}}(\xxxi)\Big|&\lesssim \Vert \Omega\Vert_{L^q(\mathbb{S}^{mn-1})}\min\big\{2^{\mu |\alpha|},2^{\mu(mn-\delta)} \big\}, \qq 1\le |\alpha|\le mn.
 \end{align*}
Finally, we have
 \begin{equation}\label{kmuhatest}
  \Big| \sum_{\gamma\in\bbz}\partial^{\alpha}\wh{K_{\mu}^{\gamma}}(\xxxi)\Big|\lesssim \Vert \Omega\Vert_{L^q(\mathbb{S}^{mn-1})}
2^{(1-\delta){\mu}}, \q  {\mu}\le 0
 \end{equation}
 for all multi-indices $\alpha$ with $|\alpha|\le mn$.
The above inequalities play a key role in establishing the boundedness of $\LL_{\Omega}$ in Theorem \ref{knownbdresult}. More precisely,  a multilinear Mihlin's multiplier theory in \cite{Co_Me1978_multi, Gr_To2002}, together with the second estimate in \eqref{kmuhatest}, implies 
\begin{align*}
&\bigg\Vert \sum_{\mu\le 0}\sum_{\gamma\in\bbz}T_{K_{\mu}^{\gamma}}\big(f_1,\dots,f_m \big) \bigg\Vert_{L^p(\bbrn)}\\
&\le \bigg( \sum_{\mu\le 0}\Big\Vert \sum_{\gamma\in\bbz}T_{K_{\mu}^{\gamma}}\big(f_1,\dots,f_m \big) \Big\Vert_{L^p(\bbrn)}^{\min{(1,p)}} \bigg)^{\frac{1}{\min{(1,p)}}}\\
&\lesssim \bigg( \sum_{\mu\le 0}\Big(2^{(1-\delta)\mu }\Vert \Omega\Vert_{L^q(\mathbb{S}^{mn-1})}\prod_{j=1}^{m}\Vert f_j\Vert_{L^{p_j}(\bbrn)}\Big)^{\min{(1,p)}}\bigg)^{\frac{1}{\min{(1,p)}}}\\
&\lesssim   \Vert \Omega\Vert_{L^q(\mathbb{S}^{mn-1})}\prod_{j=1}^{m}\Vert f_j\Vert_{L^{p_j}(\bbrn)}.
\end{align*}
When $\mu> 0$,  a wavelet decomposition method with the estimate \eqref{khat0pakh0est} yields that
\begin{equation}\label{initialestimate}
\Big\Vert \sum_{\gamma\in\bbz}T_{K_{\mu}^{\gamma}}\big(f_1,\dots,f_m \big) \Big\Vert_{L^{\frac{2}{m}}(\bbrn)}\lesssim 2^{-\epsilon_0\mu}\Vert \Omega\Vert_{L^{q}(\mathbb{S}^{mn-1})}\prod_{j=1}^{m}\Vert f_j\Vert_{L^{2}(\bbrn)}
\end{equation}
for some $\epsilon_0>0$ and any $q>\frac{2m}{m+1}$. Later, the estimate \eqref{initialestimate} has been improved and extended to general $1<p_1,\dots,p_m<\infty$ through multilinear interpolation methods. We refer to \cite{Do_Sl2024,Gr_He_Ho_Park_JLMS} for more details.
This is also a central idea in the proof of Theorem \ref{maximalinitial} and we will carry out similar arguments.\\

\subsection{Reduction}
Let $1<q<\infty$.
We recall the maximal  operators $\mathcal{M}_{\Omega}$ and $\LL_{\Omega}^{\sharp}$ are given by
\begin{equation*}
\mathcal M_\Om(f_1,\dots, f_m)(x)=
\sup_{R>0}  \f{1}{R^{mn} } \idotsint\limits_{|\yyy|\le R}   |\Om (\yyy' ) | 
\prod_{j=1}^m \big|  f_j(x-y_j) \big|  ~ d\yyy
\end{equation*}
and
\begin{equation*}
\LL_{\Om}^{\sharp}\big(f_1,\dots,f_m\big)(x) = 
\sup_{\tau\in \mathbb Z} \Big| \sum_{\gamma<\tau} \sum_{\mu\in\bbz}
T_{K_{\mu}^{\gamma}}\big(f_1,\dots,f_m\big)(x)\Big|.
\end{equation*}
As mentioned in Section \ref{introsec}, it is known in \cite{Gr_He_Ho_Park2024} that
\begin{equation*}
\LL_{\Om}^*\big(f_1,\dots,f_m\big)
\le \mathcal{M}_{\Omega}\big(f_1,\dots,f_m\big)(x)+  \LL_{\Om}^{\sharp}\big(f_1,\dots,f_m\big).
\end{equation*}
The boundedness of the first maximal function $\mathcal{M}_{\Omega}\big(f_1,\dots,f_m\big)(x)$ can be treated by the following lemma.
\begin{customlemma}{F}\cite{Gr_He_Ho_Park2024}\label{maxomeest}
Let $1<p_1,\dots , p_m<\infty$ and $\frac{1}{m}<p<\infty$ with $\frac{1}{p}=\frac{1}{p_1}+\cdots +\frac{1}{p_m}$.
Suppose that $1<q\le \infty$, $\f 1p< \f{1}{q}+\f{m}{q' }$, and  $\Om \in L^q(\mathbb S^{mn-1})$.
Given $f_j\in L^{p_j}(\bbrn)$, there exists a measure zero set $E$ such that for $x\in \bbrn\setminus E$
$$\int_{|\yyy|\le R}   |\Om (\yyy' ) | 
\prod_{j=1}^m \big|  f_j(x-y_j) \big|  ~ d\yyy<\infty$$
for all $R>0$. In this case, 
$$\text{the maximal function $\mathcal M_{\Om}(f_1,\dots, f_m)$ is well-defined on $\bbrn\setminus E$} $$
and
\begin{equation*}
\big\|\mathcal M_{\Om}(f_1,\dots, f_m)\big\|_{L^{p}(\bbrn)}\lesssim_q \|\Om\|_{L^q(\mathbb S^{mn-1})}
\prod_{j=1}^{m}\|f_j\|_{L^{p_j}(\bbrn)}
\end{equation*}
for functions $f_j\in L^{p_j}(\bbrn)$.
\end{customlemma}
Note that the condition $$\frac{1}{p}<\frac{1}{q}+\frac{m}{q'}$$
is equivalent to
$$\Big(\frac{1}{p_1},\dots,\frac{1}{p_m}\Big)\in \HH^m_{J_m}\Big(\frac{1}{q'}\Big),$$
and thus Lemma~\ref{maxomeest} yields
$$\big\|\mathcal M_{\Om}(f_1,\dots, f_m)\big\|_{L^{p}(\bbrn)}\lesssim_s \|\Om\|_{L^{\frac{1}{1-s}}(\mathbb S^{mn-1})}
\prod_{j=1}^{m}\|f_j\|_{L^{p_j}(\bbrn)}
$$ provided that
$$\Big(\frac{1}{p_1},\dots,\frac{1}{p_m}\Big)\in\HH^m(s).$$
Therefore, it remains to establish the boundedness of $\LL_{\Om}^{\sharp}$. 
For this one, we write
\begin{equation*}
  \LL_{\Om}^{\sharp}\big(f_1,\dots,f_m\big)  
  \le  \sup_{\tau \in \mathbb Z} \Big| 
  {  \sum_{ \gamma<\tau} \sum_{\mu\le 0} T_{K_{\mu}^{\gamma}}\big(f_1,\dots,f_m\big)}
  \Big|  + \sum_{\mu> 0} 
  \LL_{\Om,\mu}^{\sharp}\big(f_1,\dots,f_m\big) 
  \end{equation*}
where we recall
$$
  \LL_{\Om,\mu}^{\sharp}\big(f_1,\dots,f_m\big)(x) 
  = \sup_{\tau\in \mathbb Z} \Big|   \sum_{\gamma<\tau} T_{K_{\mu}^{\gamma}}\big(f_1,\dots,f_m\big)  (x)  \Big| .
  $$
  In addition, it has been already verified in \cite[Proposition 4.1]{Gr_He_Ho_Park2024} that
  \begin{equation*}
  \bigg\Vert\sup_{\tau \in \mathbb Z} \Big| 
  {  \sum_{ \gamma<\tau} \sum_{\mu \le 0} T_{K_{\mu}^{\gamma}}\big(f_1,\dots,f_m\big)}
  \Big|  \bigg\Vert_{L^p(\bbrn)}\lesssim_q \Vert \Omega\Vert_{L^q(\mathbb{S}^{mn-1})}\prod_{j=1}^{m}\Vert f_j\Vert_{L^{p_j}(\bbrn)}.
  \end{equation*}
Consequently, matters reduce to
 \begin{equation}\label{mainkeyest}
\bigg\|   \sum_{\mu> 0} \LL_{\Om,\mu}^{\sharp}\big(f_1,\dots,f_m\big) \bigg\|_{L^{p}(\bbrn) } \lesssim_{s}  \| \Om\|_{L^{\frac{1}{1-s}}(\mathbb{S}^{mn-1})}\prod_{j=1}^{m}\Vert f_j\Vert_{L^{p_j}(\bbrn)}.
\end{equation}
We will actually prove that there exists $\epsilon_0>0$ such that
 \begin{equation}\label{mainkeyest22}
\big\|  \LL_{\Om,\mu}^{\sharp}\big(f_1,\dots,f_m\big) \big\|_{L^{p}(\bbrn) } \lesssim_{s,\epsilon_0}  2^{-\epsilon_0 \mu}\| \Om\|_{L^{\frac{1}{1-s}}(\mathbb{S}^{mn-1})}\prod_{j=1}^{m}\Vert f_j\Vert_{L^{p_j}(\bbrn)}, \q \mu> 0,
\end{equation}
which finally deduces \eqref{mainkeyest}.\\

\subsection{Proof of \eqref{mainkeyest22}}
It is known in \cite{Gr_He_Ho_Park2024} that for $\frac{2m}{m+1}<q\le \infty$, there exists $\delta_0>0$ such that
\begin{equation}\label{multiinitialest}
\big\| \LL_{\Om,\mu}^{\sharp}(f_1,\dots,f_m )\big\|_{L^{\frac{2}{m}}(\bbrn)}\lesssim_{\delta_0} 2^{-\delta_0\mu} \|\Om\|_{L^q(\mathbb S^{mn-1})}
\prod_{j=1}^{m}\|f_j\|_{L^{2}(\bbrn)}, \q \mu> 0.
\end{equation}
For general $1<p_1,\dots,p_m<\infty$, we will prove the following two propositions.
\begin{proposition}\label{keyproposharp}
Let $1<p,p_1,\dots,p_m<\infty$ and $\frac{1}{p}=\frac{1}{p_1}+\cdots+\frac{1}{p_m}$.
Suppose that $\mu\in\bbn$ and $\Omega\in L^1(\mathbb{S}^{mn-1})$.
Then there exist  constants $M>0$ and $C_M>0$ such that
\begin{equation*}
\big\| \LL_{\Om,\mu}^{\sharp}(f_1,\dots,f_m )\big\|_{L^{p}(\bbrn)}\le C_M \mu^{M} \|\Om\|_{L^1(\mathbb S^{mn-1})}
\prod_{j=1}^{m}\|f_j\|_{L^{p_j}(\bbrn)}
\end{equation*}
for Schwartz functions $f_1,\dots,f_m$ on $\bbrn$.
\end{proposition}

\begin{proposition}\label{hsmainestpropo}
Let $0< s\le 1$, $\frac{1}{m}<p<\infty$, and $1<p_1,\dots,p_m<\infty$ with $\frac{1}{p}=\frac{1}{p_1}+\cdots+\frac{1}{p_m}$.
Suppose that $\mu\in\bbn$,
$\big(\frac{1}{p_1},\dots,\frac{1}{p_m} \big)\in \mathcal{H}^m(s)$,
and $\Omega\in L^{\frac{1}{1-s}}(\mathbb{S}^{mn-1})$ with \eqref{vanishingmtcondition}.
Then for any $\epsilon>0$, there exists a constant $C_{\epsilon}>0$ such that
\begin{equation}\label{keypropoesttri}
\big\| \LL_{\Om,\mu}^{\sharp}(f_1,\dots,f_m )\big\|_{L^{p}(\bbrn)}\le C_{\epsilon}2^{ \epsilon\mu} \|\Om\|_{L^{\frac{1}{1-s}}(\mathbb S^{mn-1})}
\prod_{j=1}^{m}\|f_j\|_{L^{p_j}(\bbrn)}
\end{equation}
for Schwartz functions $f_1,\dots,f_m$ on $\bbrn$.
\end{proposition}
The proof of Propositions \ref{keyproposharp} and \ref{hsmainestpropo} will be given in Sections \ref{pfkeypro1} and \ref{pfkeypro2}, respectively.\\

We note that 
$$\Vert \Omega\Vert_{L^q(\mathbb{S}^{mn-1})}\lesssim \Vert \Omega\Vert_{L^{\infty}(\mathbb{S}^{mn-1})} \q \text{ for all }~ 1\le q<\infty$$
and thus Proposition \ref{hsmainestpropo} deduces that for any $\epsilon>0$ and  $1<p_1,\dots,p_m<\infty$ with $\frac{1}{p}=\frac{1}{p_1}+\cdots+\frac{1}{p_m}$, 
\begin{equation}\label{generalrgest}
\big\| \LL_{\Om,\mu}^{\sharp}(f_1,\dots,f_m )\big\|_{L^{p}(\bbrn)}\lesssim_{\epsilon} 2^{ \epsilon\mu} \|\Om\|_{L^{\infty}(\mathbb S^{mn-1})}
\prod_{j=1}^{m}\|f_j\|_{L^{p_j}(\bbrn)}.
\end{equation}
Interpolating this estimate with the initial estimate \eqref{multiinitialest}, we obtain, via Lemma \ref{interpollemma}, that 
\begin{equation}\label{generalrgestimp}
\big\| \LL_{\Om,\mu}^{\sharp}(f_1,\dots,f_m )\big\|_{L^{p}(\bbrn)}\lesssim_{\delta_1} 2^{ -\delta_1\mu} \|\Om\|_{L^{\infty}(\mathbb S^{mn-1})}
\prod_{j=1}^{m}\|f_j\|_{L^{p_j}(\bbrn)}
\end{equation}
for some $\delta_1>0$. Here, the exponential decay $2^{-\delta_1 \mu}$ could be achieved due to the arbitrarily slow growth in \eqref{generalrgest} while the estimate \eqref{multiinitialest} has a fixed exponential decay in $\mu$. 

Now we introduce a method to improve the $L^{\infty}$ norm of $\Omega$ in \eqref{generalrgestimp} to $L^{\frac{1}{1-s}}$ norm so that \eqref{mainkeyest22} is established.
Suppose that $0<s<1$,
$\big( \frac{1}{p_1},\cdots,\frac{1}{p_m}\big)\in\HH^m(s)$,
and
$$\Vert \Omega\Vert_{L^{\frac{1}{1-s}}(\mathbb{S}^{mn-1})}=\Vert f_1\Vert_{L^{p_1}(\bbrn)}=\cdots=\Vert f_m\Vert_{L^{p_m}(\bbrn)}=1.$$
Then it is sufficient to show the existence of $\epsilon_0>0$ for which
 \begin{equation}\label{mainkeyest33}
\big\|  \LL_{\Om,\mu}^{\sharp}\big(f_1,\dots,f_m\big) \big\|_{L^{p}(\bbrn) } \lesssim_{\epsilon_0}  2^{-\epsilon_0 \mu}, \q \mu\in\bbn.
\end{equation}
For this one, we first decompose the sphere $\mathbb{S}^{mn-1}$ as
$$ \mathbb{S}^{mn-1}=\dot{\bigcup_{l\in\bbn_0}}D^l$$
where
$$D^{l}:=\begin{cases} \big\{\theta\in\mathbb{S}^{mn-1}: |\Omega(\theta)|\le 1\big\} & \text{ if } l=0\\
\big\{\theta\in\mathbb{S}^{mn-1}:  2^{l-1}< |\Omega(\theta)|\le 2^l    \big\}  & \text{ if } l\ge 1 \end{cases}, $$
and write
\begin{equation}\label{omegadecomp}
\Omega(\theta)=\Omega(\theta)-\int_{\bbs^{mn-1}}\Omega(\eta)     \;d\sigma(\eta) = \sum_{l=0}^{\infty}\bigg( \Omega(\theta)\chi_{D^{l}}(\theta)-\int_{D^{l}} \Omega(\eta) \;d\sigma(\eta)\bigg)=:\sum_{l=0}^{\infty}\Omega^{l}(\theta).
\end{equation}
Then the left-hand side of \eqref{mainkeyest33} is bounded by
\begin{equation}\label{lastdecomex}
\bigg( \sum_{l\in\bbn_0} \big\|   \LL_{\Om^l,\mu}^{\sharp}\big(f_1,\dots,f_m\big) \big\|_{L^{p}(\bbrn) }^{\min{(1,p)}} \bigg)^{\frac{1}{\min{(1,p)}}}.
\end{equation}
We note that each $\Omega^{l}$ satisfies the vanishing moment condition 
$$\int_{\bbs^{mn-1}}\Omega^{l}(\theta)\;d\sigma(\theta)=0$$
and thus we can apply \eqref{generalrgestimp}, Propositions \ref{keyproposharp} and \ref{hsmainestpropo} to $\Omega^l$ instead of $\Omega$.
Obviously,
\begin{equation*}
\Vert \Omega^{l}\Vert_{L^{\infty}(\bbs^{mn-1})}\le 2^{l+1}
\end{equation*}	
and thus \eqref{generalrgestimp} yields that
\begin{equation}\label{commonest}
\big\|   \LL_{\Om^l,\mu}^{\sharp}\big(f_1,\dots,f_m\big) \big\|_{L^{p}(\bbrn) }\lesssim 2^{-\delta_1 \mu}\big\Vert \Omega^l\big\Vert_{L^{\infty}(\bbrn)}\lesssim 2^{-\delta_1 \mu}2^l.
\end{equation}

When $p> 1$,  we see
\begin{equation}\label{omll1s}
\big\|\Omega^l\big\|_{L^1(\bbs^{mn-1})} \le2\int_{D^l}\big|\Omega(\theta)\big| \;d\sigma(\theta)\lesssim_s 2^{-l(\frac{1}{1-s}-1)}\int_{D^l}\big| \Omega(\theta)\big|^{\frac{1}{1-s}}\;d\sigma(\theta)\le 2^{-l\frac{s}{1-s}}
\end{equation}
and thus Proposition \ref{keyproposharp} deduces
\begin{equation}\label{commonestpge1}
\big\| \LL_{\Om^l,\mu}^{\sharp}(f_1,\dots,f_m )\big\|_{L^{p}(\bbrn)}\lesssim_M \mu^{M} \big\|\Om^l\big\|_{L^1(\mathbb S^{mn-1})}\lesssim 2^{-l\frac{s}{1-s}}\mu^M
\end{equation} for some $M>0$.
We choose $1-s<\eta<1$, or consequently,
$$\eta\Big(\frac{s}{1-s}\Big)-(1-\eta)>0,$$
and by averaging \eqref{commonest} and \eqref{commonestpge1}, we obtain
\begin{align*}
\big\Vert \LL_{\Omega^l,\mu}^{\sharp}(f_1,\dots,f_m)\big\Vert_{L^p(\bbrn)}&\lesssim  \big(\mu^M    2^{-l\frac{s}{1-s}} \big)^{\eta}\big( 2^{-\delta_1 \mu}2^l     \big)^{1-\eta} =\mu^{M\eta}2^{-\delta_1(1-\eta)\mu} 2^{-l(\eta(\frac{s}{1-s})-(1-\eta))}.
\end{align*}
Clearly, the right-hand side is summable over $l\in\bbn_0$ and thus  \eqref{lastdecomex} is dominated by a constant times
$$\mu^{M\eta}2^{-\delta_1(1-\eta)\mu}\Big(\sum_{l\in\bbn_0} 2^{-l(\eta(\frac{s}{1-s})-(1-\eta) )}\Big)\lesssim 2^{-\epsilon_0 \mu}, \q \mu> 0$$
for some $\epsilon_0>0$,
as desired.

Now assume that $\frac{1}{m}<p\le 1$.
In this case, we note that $$\bigcup_{0<r<s}\HH^m(r)=\HH^m(s)$$
and thus there exists $0<r<s$ such that
$$\Big(\frac{1}{p_1},\cdots,\frac{1}{p_m} \Big)\in \HH^m(r).$$
Choosing 
\begin{equation}\label{eprange}
0<\epsilon<\delta_1\Big( \frac{s-r}{1-s}\Big), \q \text{ or equivalently }\q 0<\frac{\epsilon}{\delta_1}<\frac{s-r}{1-s}
\end{equation}
and applying Proposition \ref{hsmainestpropo} to $\LL_{\Omega^l,\mu}^{\sharp}$, 
we have
\begin{equation*}
\big\| \LL_{\Om^l,\mu}^{\sharp}(f_1,\dots,f_m )\big\|_{L^{p}(\bbrn)}\lesssim_{\epsilon} 2^{ \epsilon\mu} \big\|\Om^l\big\|_{L^{\frac{1}{1-r}}(\mathbb S^{mn-1})}.
\end{equation*}
Similar to \eqref{omll1s}, we can estimate
\begin{align*}
\big\Vert \Omega^l\big\Vert_{L^{\frac{1}{1-r}}(\mathbb{S}^{mn-1})} &\lesssim \bigg(\int_{D^l}\big| \Omega(\theta)\big|^{\frac{1}{1-r}}\; d\sigma(\theta) \bigg)^{1-r}\lesssim_s 
\bigg(\int_{D^l}    2^{-l(\frac{1}{1-s}-\frac{1}{1-r})} \big| \Omega(\theta)\big|^{\frac{1}{1-s}}\; d\sigma(\theta) \bigg)^{1-r}\\
&\lesssim 2^{-l(\frac{1-r}{1-s}-1)}=2^{-l(\frac{s-r}{1-s})}
\end{align*}
and this yields
\begin{equation}\label{commonestple1}
\big\| \LL_{\Om^l,\mu}^{\sharp}(f_1,\dots,f_m )\big\|_{L^{p}(\bbrn)}\lesssim_{\epsilon} 2^{ \epsilon\mu} 2^{-l(\frac{s-r}{1-s})}.
\end{equation}
Now we choose $0<\eta<1$ with $\frac{1-s}{1-r}<\eta<\frac{\delta_1}{\delta_1+\epsilon}$ (possibly due to \eqref{eprange}) so that 
$$\delta_1(1-\eta)-\epsilon \eta>0    \q \text{ and }\q   \eta\Big(\frac{s-r}{1-s}\Big)-(1-\eta)>0$$ 
and average the estimates \eqref{commonest} and \eqref{commonestple1} to obtain
$$\big\| \LL_{\Om^l,\mu}^{\sharp}(f_1,\dots,f_m )\big\|_{L^{p}(\bbrn)}\lesssim \big(   2^{-\delta_1 \mu}2^l   \big)^{1-\eta}\big(  2^{ \epsilon\mu} 2^{-l\frac{s-r}{1-s}}   \big)^{\eta}
=2^{-\mu(\delta_1(1-\eta)-\epsilon \eta)}2^{-l( \eta (\frac{s-r}{1-s})-(1-\eta)   )}.$$
Finally, \eqref{lastdecomex} is bounded by a constant multiple of 
$$2^{-\mu(\delta_1(1-\eta)-\epsilon \eta)}\bigg( \sum_{l\in\bbn_0}   2^{-lp( \eta (\frac{s-r}{1-s})-(1-\eta)   )}    \bigg)^{\frac{1}{p}}\sim 2^{-\mu(\delta_1(1-\eta)-\epsilon \eta)}.$$
By taking $\epsilon_0=\delta_1(1-\eta)-\epsilon \eta>0$,
we complete the proof of \eqref{mainkeyest33}.

\hfill

\section{Proof of Proposition \ref{keyproposharp}}\label{pfkeypro1}
Without loss of generality, we may assume 
$$\Vert f_1\Vert_{L^{p_1}(\bbrn)}=\cdots=\Vert f_m\Vert_{L^{p_m}(\bbrn)}=\Vert \Omega\Vert_{L^1(\mathbb{S}^{mn-1})}= 1.$$
We first employ Littlewood-Paley decompositions for each $f_j$ so that
$$\sum_{\gamma<\tau}T_{K_{\mu}^{\gamma}}\big(f_1,\dots,f_m\big)(x)=\sum_{\gamma<\tau}\sum_{k_1,\dots,k_m\in\bbz}T_{K_{\mu}^{\gamma}}\big(\psi_{k_1}\ast f_1,\dots,\psi_{k_m}\ast f_m\big)(x)    $$
and this can be written, in view of Lemma \ref{yyvt}, as a finite sum of form
\begin{equation*}
\sum_{\gamma<\tau} \Phi_{\mu+\gamma}^{m+1}\ast T_{\mu}^{\gamma}\big(f_1,\dots,f_m\big)(x)
\end{equation*}
where
$$T_{\mu}^{\gamma}\big(f_1,\dots,f_m\big)(x):= T_{K_{\mu}^{\gamma}}\big(\Phi_{\mu+\gamma}^1\ast f_1,\dots,\Phi_{\mu+\gamma}^m\ast f_m\big)(x).$$
Therefore, it suffices to show that there exists $M>0$ such that
\begin{equation}\label{lpfinalgoal}
\bigg\Vert \sup_{\tau\in\bbz}\Big| \sum_{\gamma<\tau} \Phi_{\mu+\gamma}^{m+1}\ast T_{\mu}^{\gamma}\big(f_1,\dots,f_m\big) \Big|\bigg\Vert_{L^p(\bbrn)}\lesssim_M \mu^{M}.
\end{equation}
Note that at least two of $\widehat{\Phi^1}$, $\widehat{\Phi^2}$, $\cdots$, $\widehat{\Phi^{m+1}}$  are compactly supported in an annulus, and 
the inequality \eqref{lpfinalgoal} will be achieved separately depending on whether the last one $\widehat{\Phi^{m+1}}$ is supported in an annulus or not.
One of the key estimates for both cases is that for any $M>0$ 
\begin{equation}\label{Omegal1est}
\int_{(\bbrn)^m}\big( \ln{(e+|\yyy|)}\big)^M\big|K_{\mu}^0(\yyy) \big| \; d\yyy\lesssim_M \Vert \Omega\Vert_{L^1(\mathbb{S}^{mn-1})}=1
\end{equation}
which is known in \cite[page 2267]{Do_Sl2024}.

\hfill

{\bf Case 1.} Suppose that $\widehat{\Phi^{m+1}}$ is supported in an annulus.
In this case, we may assume $\widehat{\Phi^{1}}$ is also supported in an annulus, as the other cases follow in a symmetric way.

We first claim 
\begin{equation}\label{suptauinz}
\bigg\Vert \sup_{\tau\in\bbz}\Big| \sum_{\gamma<\tau} \Phi_{\mu+\gamma}^{m+1}\ast T_{\mu}^{\gamma}\big(f_1,\dots,f_m\big) \Big|\bigg\Vert_{L^p(\bbrn)}\lesssim \bigg\Vert \Big(\sum_{\gamma\in\bbz} \big|T_{\mu}^{\gamma}\big(f_1,\dots,f_m\big) \big|^2 \Big)^{\frac{1}{2}}\bigg\Vert_{L^p(\bbrn)}.
\end{equation} 
To verify this, we observe that the Fourier transform of $ \sum_{\gamma<\tau} \Phi_{\mu+\gamma}^{m+1}\ast T_{\mu}^{\gamma}\big(f_1,\dots,f_m\big)$ is supported in a ball of radius $C 2^{\mu+\tau}$, centered at the origin, for some $C>0$ and thus it can be written as
\begin{align*}
 &\sum_{\gamma<\tau} \Phi_{\mu+\gamma}^{m+1}\ast T_{\mu}^{\gamma}\big(f_1,\dots,f_m\big)=\Lambda_{\mu+\tau}\ast  \Big(\sum_{\gamma<\tau} \Phi_{\mu+\gamma}^{m+1}\ast T_{\mu}^{\gamma}\big(f_1,\dots,f_m\big)\Big)\\
 &\quad=\Lambda_{\mu+\tau}\ast \Big( \sum_{\gamma\in\bbz} \Phi_{\mu+\gamma}^{m+1}\ast T_{\mu}^{\gamma}\big(f_1,\dots,f_m\big)\Big)- \Lambda_{\mu+\tau}\ast \Big( \sum_{\gamma\ge \tau} \Phi_{\mu+\gamma}^{m+1}\ast T_{\mu}^{\gamma}\big(f_1,\dots,f_m\big)\Big)
\end{align*}
where $\Lambda_{\mu+\tau}$ is a radial Schwartz function on $\bbrn$ whose Fourier transform is equal to $1$ on the ball $B(0,C2^{\mu+\tau})$ and is supported in a larger ball of radius $\tilde{C}2^{\mu+\tau}$ for some $\tilde{C}>C$.
Therefore, the left-hand side of \eqref{suptauinz} is bounded by the sum of
$$\mathcal{I}_1^{\mu}:=\bigg\Vert  \sup_{\tau\in\bbz}\Big|\Lambda_{\mu+\tau}\ast \Big(   \sum_{\gamma\in\bbz} \Phi_{\mu+\gamma}^{m+1}\ast T_{\mu}^{\gamma}\big(f_1,\dots,f_m\big)   \Big) \Big|       \bigg\Vert_{L^p(\bbrn)}$$
and
$$\mathcal{I}_2^{\mu}:=\bigg\Vert  \sup_{\tau\in\bbz}\Big|    \Lambda_{\mu+\tau}\ast \Big( \sum_{\gamma\ge \tau} \Phi_{\mu+\gamma}^{m+1}\ast T_{\mu}^{\gamma}\big(f_1,\dots,f_m\big)\Big)      \Big|      \bigg\Vert_{L^p(\bbrn)}.$$

Using \eqref{phikfptest}, the $L^p$ boundedness for $\mathcal{M}$, and \eqref{ltchaest}, we have
\begin{align*}
\mathcal{I}_1^{\mu}&\lesssim \bigg\Vert \mathcal{M}\Big(   \sum_{\gamma\in\bbz} \Phi_{\mu+\gamma}^{m+1}\ast T_{\mu}^{\gamma}\big(f_1,\dots,f_m\big)   \Big) \bigg\Vert_{L^p(\bbrn)}\lesssim \bigg\Vert    \sum_{\gamma\in\bbz} \Phi_{\mu+\gamma}^{m+1}\ast T_{\mu}^{\gamma}\big(f_1,\dots,f_m\big)   \bigg\Vert_{L^p(\bbrn)}\\
&\lesssim \bigg\Vert \Big( \sum_{\gamma\in\bbz}\big|  \Phi_{\mu+\gamma}^{m+1}\ast T_{\mu}^{\gamma}\big(f_1,\dots,f_m\big) \big|^2    \Big)^{\frac{1}{2}}\bigg\Vert_{L^p(\bbrn)}\lesssim \bigg\Vert \Big( \sum_{\gamma\in\bbz}\Big|  \mathcal{M}\Big( T_{\mu}^{\gamma}\big(f_1,\dots,f_m\big)\Big) \Big|^2    \Big)^{\frac{1}{2}}\bigg\Vert_{L^p(\bbrn)}\\
&\lesssim \bigg\Vert \Big( \sum_{\gamma\in\bbz}\big|  T_{\mu}^{\gamma}\big(f_1,\dots,f_m\big)\big|^2    \Big)^{\frac{1}{2}}\bigg\Vert_{L^p(\bbrn)}
\end{align*}
 where we recall the Fourier transform of  $\Phi_{\mu+\gamma}^{m+1}$ is supported in an annulus of size $2^{\mu+\gamma}$.
 
To estimate $\mathcal{I}^{\mu}_2$, we note that 
$\wh{\Lambda_{\mu+\tau}}$ is supported in a ball of radius $\wt{C}2^{\mu+\tau}$ while
$\wh{\Phi_{\mu+\gamma}^{m+1}}$ is in an annulus of size $2^{\mu+\gamma}$. Hence, there is a positive integer $C_0$ such that
$$\Lambda_{\mu+\tau}\ast \Phi_{\mu+\gamma}^{m+1}\ast T_{\mu}^{\gamma}(f_1,\dots,f_m)=0\quad \text{unless }~\gamma\le \tau +C_0.$$
This yields that
\begin{align*}
\mathcal{I}_2^{\mu}&=\bigg\Vert  \sup_{\tau\in\bbz}\Big|   \Lambda_{\mu+\tau}\ast \Big( \sum_{\gamma= \tau}^{\tau+C_0} \Phi_{\mu+\gamma}^{m+1}\ast T_{\mu}^{\gamma}\big(f_1,\dots,f_m\big)\Big)      \Big|      \bigg\Vert_{L^p(\bbrn)}\\
&=\bigg\Vert  \sup_{\tau\in\bbz}\Big|  \Lambda_{\mu+\tau}\ast \Big( \sum_{\gamma= 0}^{C_0} \Phi_{\mu+\tau+\gamma}^{m+1}\ast T_{\mu}^{\gamma+\tau}\big(f_1,\dots,f_m\big)\Big)      \Big|      \bigg\Vert_{L^p(\bbrn)}\\
&\le \sum_{\gamma=0}^{C_0} \bigg\Vert  \sup_{\tau\in\bbz}\Big|   \Lambda_{\mu+\tau}\ast   \Phi_{\mu+\tau+\gamma}^{m+1}\ast T_{\mu}^{\gamma+\tau}\big(f_1,\dots,f_m\big)      \Big|      \bigg\Vert_{L^p(\bbrn)}\\
&\le \sum_{\gamma=0}^{C_0} \bigg\Vert \Big(  \sum_{\tau\in\bbz}    \Big|   \Lambda_{\mu+\tau}\ast   \Phi_{\mu+\tau+\gamma}^{m+1}\ast T_{\mu}^{\gamma+\tau}\big(f_1,\dots,f_m\big) \Big|^2       \Big)^{\frac{1}{2}}       \bigg\Vert_{L^p(\bbrn)}.
\end{align*}
Now using \eqref{phikfptest} and \eqref{hlmax}, the preceding expression is bounded by a constant times
\begin{align*}
 &\sum_{\gamma=0}^{C_0} \bigg\Vert \Big(  \sum_{\tau\in\bbz}    \Big|  \mathcal{M}\Big(  T_{\mu}^{\gamma+\tau}\big(f_1,\dots,f_m\big) \Big)\Big|^2       \Big)^{\frac{1}{2}}       \bigg\Vert_{L^p(\bbrn)}\\
 &=(C_0+1)\bigg\Vert \Big(  \sum_{\tau\in\bbz}    \Big|  \mathcal{M}\Big(   T_{\mu}^{\tau}\big(f_1,\dots,f_m\big) \Big)\Big|^2       \Big)^{\frac{1}{2}}       \bigg\Vert_{L^p(\bbrn)}\\
 &\lesssim \bigg\Vert \Big(  \sum_{\tau\in\bbz}    \big|    T_{\mu}^{\tau}\big(f_1,\dots,f_m\big) \big|^2       \Big)^{\frac{1}{2}}       \bigg\Vert_{L^p(\bbrn)},
\end{align*}
which completes the proof of the claim \eqref{suptauinz}.

Now we need to prove that 
\begin{equation}\label{2ndkeyest}
\bigg\Vert \Big(\sum_{\gamma\in\bbz} \big|T_{\mu}^{\gamma}\big(f_1,\dots,f_m\big) \big|^2 \Big)^{\frac{1}{2}}\bigg\Vert_{L^p(\bbrn)}\lesssim_M \mu^M 
\end{equation} for some $M>0$.
Applying \eqref{kernelcharacter} and performing a change of variables,
\begin{align}
T_{\mu}^{\gamma}\big(f_1,\dots,f_m\big)(x)&=\int_{(\bbrn)^m} 2^{\gamma mn}K_{\mu}^{0}(2^{\gamma}\yyy)\prod_{j=1}^{m}\Phi_{\mu+\gamma}^{j}\ast f_j(x-y_j)        \; d\yyy \nonumber\\
&=\int_{(\bbrn)^m} K_{\mu}^{0}(\yyy)\prod_{j=1}^{m}\Phi_{\mu+\gamma}^{j}\ast f_j(x-2^{-\gamma}y_j)\; d\yyy \label{tmugaptest}
\end{align}
and then Minkowski's inequality yields that the left-hand side of \eqref{2ndkeyest} is bounded by
\begin{equation}\label{switchest}
\int_{(\bbrn)^m}\big| K_{\mu}^{0}(\yyy)\big| \bigg\Vert  \bigg( \sum_{\gamma\in\bbz}\Big|   \prod_{j=1}^{m}  \Phi_{\mu+\gamma}^{j}\ast f_j(\cdot-2^{-\gamma}y_j)          \Big|^2   \bigg)^{\frac{1}{2}}  \bigg\Vert_{L^p(\bbrn)}    \; d\yyy.
\end{equation}
The $L^p$ norm would be
\begin{align*}
&\bigg\Vert  \bigg( \sum_{\gamma\in\bbz}\Big|   \prod_{j=1}^{m}  \Phi_{\gamma}^{j}\ast f_j(\cdot-2^{-\gamma+\mu}y_j)          \Big|^2   \bigg)^{\frac{1}{2}}  \bigg\Vert_{L^p(\bbrn)} \\
&\le \bigg\Vert \bigg( \sum_{\gamma\in\bbz}\big| \Phi_{\gamma}^{1}\ast f_1(\cdot-2^{-\gamma+\mu}y_1)\big|^2    \bigg)^{\frac{1}{2}} \bigg( \prod_{j=2}^{m}\sup_{\gamma\in\bbz}\big| \Phi_{\gamma}^j\ast f_j(\cdot-2^{-\gamma+\mu}y_j)\big| \bigg)   \bigg\Vert_{L^p(\bbrn)}
\end{align*}
and then this is no more than
\begin{align*}
&\bigg\Vert       \bigg( \sum_{\gamma\in\bbz}\big| \Phi_{\gamma}^{1}\ast f_1(\cdot-2^{-\gamma+\mu}y_1)\big|^2    \bigg)^{\frac{1}{2}}      \bigg\Vert_{L^{p_1}(\bbrn)}\prod_{j=2}^{m}\Big\Vert \sup_{\gamma\in\bbz} \big| \Phi_{\gamma}^j\ast f_j(\cdot-2^{-\gamma+\mu}y_j)\big|     \Big\Vert_{L^{p_j}(\bbrn)}\\
&\lesssim    \big(\ln{(e+2^{\mu}|y_1|)} \big)^{|\frac{1}{p_1}-\frac{1}{2}|} \Big(\prod_{j=2}^{m}\big(\ln{(e+2^{\mu}|y_j|)}\big)^{\frac{1}{p_j}} \Big)   \\
&\lesssim \mu^{|\frac{1}{p_1}-\frac{1}{2}|+\sum_{j=2}^{m}\frac{1}{p_j}} \big( \ln{(e+|\yyy|)} \big)^{|\frac{1}{p_1}-\frac{1}{2}|+\sum_{j=2}^{m}\frac{1}{p_j}}
\end{align*}
by H\"older's inequality and Lemma \ref{shiftsquare}.
This proves \eqref{switchest} is bounded by a constant multiple of
\begin{equation}  \label{case1endest}
\mu^{|\frac{1}{p_1}-\frac{1}{2}|+\sum_{j=2}^{m}\frac{1}{p_j}}  \int_{(\bbrn)^m}  \big| K_{\mu}^{0}(\yyy)\big| \big( \ln{(e+|\yyy|)} \big)^{|\frac{1}{p_1}-\frac{1}{2}|+\sum_{j=2}^{m}\frac{1}{p_j}} \; d\yyy \lesssim \mu^{|\frac{1}{p_1}-\frac{1}{2}|+\sum_{j=2}^{m}\frac{1}{p_j}} 
\end{equation}
where the inequality follows from \eqref{Omegal1est}.
Setting $M=|\frac{1}{p_1}-\frac{1}{2}|+\sum_{j=2}^{m}\frac{1}{p_j}$,
the inequality \eqref{lpfinalgoal} follows.

\hfill

{\bf Case 2.} If $\wh{\Phi^{m+1}}$ is not supported in an annulus, then at least two of $\wh{\Phi^1},\dots,\wh{\Phi^m}$ are supported in an annulus. We will consider only the case when the two are $\wh{\Phi_1}$ and $\wh{\Phi_2}$ as a symmetric argument is applicable to the other cases.
Then \eqref{phikfpetreept} and \eqref{peetremax} yield the left-hand side of \eqref{lpfinalgoal} is bounded by
\begin{align*}
\bigg\Vert \sum_{\gamma\in\bbz}\big| \Phi_{\mu+\gamma}^{m+1}\ast T_{\mu}^{\gamma}(f_1,\dots,f_m)\big|\bigg\Vert_{L^p(\bbrn)}&=\bigg\Vert \sum_{\gamma\in\bbz}\big| \Phi_{\gamma}^{m+1}\ast T_{\mu}^{\gamma-\mu}(f_1,\dots,f_m)\big|\bigg\Vert_{L^p(\bbrn)} \\
&\lesssim_{\sigma}\bigg\Vert \sum_{\gamma\in\bbz} \mathfrak{M}_{\sigma,2^{\gamma}}\Big( T_{\mu}^{\gamma-\mu}(f_1,\dots,f_m)\Big)\bigg\Vert_{L^p(\bbrn)}\\
&\lesssim_{\sigma}  \bigg\Vert \sum_{\gamma\in\bbz} \Big| T_{\mu}^{\gamma-\mu}(f_1,\dots,f_m)\Big|\bigg\Vert_{L^p(\bbrn)}
\end{align*}
for $\sigma>n$, where we note that the Fourier transform of $T_{\mu}^{\gamma-\mu}(f_1,\dots,f_m)$ is supported in a ball of radius comparable to $2^{\gamma}$.
Using \eqref{tmugaptest} and Minkowski's inequality, the last displayed expression is controlled by
\begin{equation}\label{intkmu0l1est}
\int_{(\bbrn)^m}  \big| K_{\mu}^0(\yyy)\big| \bigg\Vert \sum_{\gamma\in\bbz}\Big| \prod_{\j=1}^{m}\Phi^{j}_{\gamma}\ast f_j(\cdot-2^{-\gamma+\mu}y_j)\Big| \bigg\Vert_{L^p(\bbrn)}     \; d\yyy.
\end{equation}
Now we bound the $L^p$ norm by
\begin{align*}
\bigg\Vert \bigg(\prod_{j=1}^{2}\Big( \sum_{\gamma\in\bbz}\big| \Phi_{\gamma}^{j}\ast f_j(\cdot-2^{-\gamma+\mu}y_j)\big|^2    \Big)^{\frac{1}{2}} \bigg)\bigg( \prod_{j=3}^{m}\sup_{\gamma\in\bbz}\big| \Phi_{\gamma}^j\ast f_j(\cdot-2^{-\gamma+\mu}y_j)\big|\bigg)\bigg\Vert_{L^p(\bbrn)}
\end{align*}
and H\"older's inequality and Lemma \ref{shiftsquare} deduce that the above expression is dominated by 
\begin{align*}
&\bigg(\prod_{j=1}^{2}\Big\Vert \Big( \sum_{\gamma\in\bbz}\big| \Phi_{\gamma}^{j}\ast f_j(\cdot-2^{-\gamma+\mu}y_j)\big|^2    \Big)^{\frac{1}{2}} \Big\Vert_{L^{p_j}(\bbrn)}\bigg) \bigg(\prod_{j=3}^{m}\Big\Vert \sup_{\gamma\in\bbz}\big| \Phi_{\gamma}^j\ast f_j(\cdot-2^{-\gamma+\mu}y_j)\big|\Big\Vert_{L^{p_j}(\bbrn)} \bigg)\\
&\lesssim       \bigg(\prod_{j=1}^{2}\big( \ln{(e+2^{\mu}y_j)}\big)^{|\frac{1}{p_j}-\frac{1}{2}|} \bigg)\bigg( \prod_{j=3}^{m}\big(\ln{(e+2^{\mu}|y_j|)}\big)^{\frac{1}{p_j}}\bigg)      \\
&\lesssim \mu^{|\frac{1}{p_1}-\frac{1}{2}|+|\frac{1}{p_2}-\frac{1}{2}|+\sum_{j=3}^{m}\frac{1}{p_j}}\big(\ln{(e+|\yyy|)}\big)^{|\frac{1}{p_1}-\frac{1}{2}|+|\frac{1}{p_2}-\frac{1}{2}|+\sum_{j=3}^{m}\frac{1}{p_j}}.
\end{align*}
Therefore, \eqref{intkmu0l1est} can be estimated by
\begin{align*}
&\mu^{|\frac{1}{p_1}-\frac{1}{2}|+|\frac{1}{p_2}-\frac{1}{2}|+\sum_{j=3}^{m}\frac{1}{p_j}} \int_{(\bbrn)^m}  \big| K_{\mu}^{0}(\yyy)\big| \big( \ln{(e+|\yyy|)} \big)^{|\frac{1}{p_1}-\frac{1}{2}|+|\frac{1}{p_2}-\frac{1}{2}|+\sum_{j=3}^{m}\frac{1}{p_j}} \; d\yyy \\
&\lesssim \mu^{|\frac{1}{p_1}-\frac{1}{2}|+|\frac{1}{p_2}-\frac{1}{2}|+\sum_{j=3}^{m}\frac{1}{p_j}},
\end{align*}
similar to \eqref{case1endest}. This finishes the proof of Proposition \ref{keyproposharp}.\\


\section{Proof of Proposition \ref{hsmainestpropo}}\label{pfkeypro2}
Let $0<s<1$ and recall $J_m=\{1,\dots,m\}$.
The proof is based on the induction argument used in \cite{Gr_He_Ho_Park_JLMS}.
In order to describe the idea,
we define
$$\mathscr{R}^m_l(s):=\{(t_1,\dots,t_m): t_l=1 \q \text{and }~ 0\le t_j<s ~\text{ for }~ j\not= l\}, \qquad l\in J_m$$
and 
$$\mathcal{C}^m(s):=\{(t_1,\dots,t_m):0<t_j<s,\q j\in J_m\}.$$ 

\begin{customclaim}{$X(s)$}
 Let $\frac{1}{m}<p<\infty$ and $(\frac{1}{p_1},\dots,\frac{1}{p_m})\in\mathcal{C}^m(s) $  with $\frac{1}{p_1}+\dots+\frac{1}{p_m}=\frac{1}{p}.$
 Suppose that $0<\epsilon<1$ and $\mu\in\bbn$. Then there exists $C_{\epsilon}>0$ such that
 $$\big\Vert \LL_{\Om,\mu}^{\sharp}(f_1,\dots,f_m)   \big\Vert_{L^{p}(\bbrn)}  \le C_{\epsilon} 2^{\epsilon \mu} \Vert \Omega\Vert_{L^{\frac{1}{1-s}}(\mathbb{S}^{mn-1})}  \prod_{j=1}^{m}\Vert f_j\Vert_{L^{p_j}(\bbrn)}.$$
\end{customclaim}
\begin{customclaim}{$Y(s)$}
Let $\frac{1}{m}<p<1$ and $(\frac{1}{p_1},\dots,\frac{1}{p_m})\in \bigcup_{l=1}^{m}\mathscr{R}^m_l(s)$ with $\frac{1}{p_1}+\dots+\frac{1}{p_m}=\frac{1}{p}$.
Suppose that $0<\epsilon<1$ and $\mu\in\bbn$. Then there exists $C_{\epsilon}>0$ such that
\begin{equation*}
\big\Vert \LL_{\Om,\mu}^{\sharp}(f_1,\dots,f_m)\big\Vert_{L^{p,\infty}(\bbrn)}\le C_{\epsilon} 2^{\epsilon \mu}  \Vert \Omega\Vert_{L^{\frac{1}{1-s}}(\mathbb{S}^{mn-1})}  \prod_{j=1}^{m}\Vert f_j\Vert_{L^{p_j}(\bbrn)}.
\end{equation*}
\end{customclaim}
\begin{customclaim}{$Z(s)$}
Let $\frac{1}{m}<p<\infty$ and $(\frac{1}{p_1},\dots,\frac{1}{p_m})\in\bigcup_{l=1}^m\mathscr{V}_l^m(s)$ with $\frac{1}{p_1}+\dots+\frac{1}{p_m}=\frac{1}{p}$, 
 where $\mathscr{V}_l^m(s)$ is defined in \eqref{defvlm}.
Suppose that $0<\epsilon<1$ and $\mu\in\bbn$. Then there exists $C_{\epsilon}>0$ such that
\begin{equation*}
\big\Vert \LL_{\Om,\mu}^{\sharp}(f_1,\dots,f_m)\big\Vert_{L^{p}(\bbrn)}\le C_{\epsilon} 2^{\epsilon \mu}  \Vert \Omega\Vert_{L^{\frac{1}{1-s}}(\mathbb{S}^{mn-1})}   \prod_{j=1}^{m}\Vert f_j\Vert_{L^{p_j}(\bbrn)}.
\end{equation*}
\end{customclaim}
\begin{customclaim}{$\Sigma(s)$}
Let $\frac{1}{m}<p<\infty$ and $(\frac{1}{p_1},\dots,\frac{1}{p_m})\in\mathcal{H}^m(s)$  with $\frac{1}{p_1}+\dots+\frac{1}{p_m}=\frac{1}{p}$.
Suppose that $0<\epsilon<1$ and $\mu\in\bbn$. Then there exists $C_{\epsilon}>0$ such that
\begin{equation*}
\big\Vert \LL_{\Om,\mu}^{\sharp}(f_1,\dots,f_m)\big\Vert_{L^{p}(\bbrn)}\le C_{\epsilon}2^{\epsilon \mu}   \Vert \Omega\Vert_{L^{\frac{1}{1-s}}(\mathbb{S}^{mn-1})}  \prod_{j=1}^{m}\Vert f_j\Vert_{L^{p_j}(\bbrn)}.
\end{equation*}
\end{customclaim}

 Please see Figure \ref{fig1} for the region where the claims hold in the trilinear setting.
 \begin{figure}[h]
\begin{tikzpicture}
\path[fill=green!5] (0-\gap,0,1.25)--(1-\gap,0,1.25)--(1-\gap,1,1.25)--(0-\gap,1,1.25)--(0-\gap,0,1.25);
\path[fill=green!5] (1-\gap,0,1.25)--(1-\gap,0,0)--(1-\gap,1,0)--(1-\gap,1,1.25)--(1-\gap,0,1.25);
\path[fill=green!5] (1-\gap,1,0)--(0-\gap,1,0)--(0-\gap,1,1.25)--(1-\gap,1,1.25)--(1-\gap,1,0);

\draw[dash pattern= { on 2pt off 1pt}](0-\gap,1,0)--(1-\gap,1,0)--(1-\gap,0,0)--(1-\gap,0,1.25)--(0-\gap,0,1.25)--(0-\gap,1,1.25)--(0-\gap,1,0);
\draw[dash pattern= { on 2pt off 1pt}](0-\gap,1,1.25)--(1-\gap,1,1.25)--(1-\gap,0,1.25);
\draw[dash pattern= { on 2pt off 1pt}](1-\gap,1,1.25)--(1-\gap,1,0);

\draw[dotted] (1-\gap,1,3)--(1-\gap,2.25,1.25)--(2.25-\gap,1,1.25)--(1-\gap,1,3);
\draw[dotted] (1-\gap,1,3)--(1-\gap,0,3)--(2.25-\gap,0,1.25)--(2.25-\gap,1,1.25);
\draw[dotted] (1-\gap,2.25,1.25)--(1-\gap,2.25,0)--(2.25-\gap,1,0)--(2.25-\gap,1,1.25);
\draw[dotted](1-\gap,1,3)--(0-\gap,1,3)--(0-\gap,2.25,1.25)--(1-\gap,2.25,1.25);
\draw[dotted] (2.25-\gap,0,1.25)--(2.25-\gap,0,0)--(2.25-\gap,1,0);
\draw[dotted] (0-\gap,2.25,1.25)--(0-\gap,2.25,0)--(1-\gap,2.25,0);
\draw[dotted] (1-\gap,0,3)--(0-\gap,0,3)--(0-\gap,1,3);

\draw[dash pattern= { on 1pt off 1pt}] (0-\gap,0,0)--(1-\gap,0,0);
\draw[dash pattern= { on 1pt off 1pt}] (0-\gap,0,0)--(0-\gap,1,0);
\draw[dash pattern= { on 1pt off 1pt}] (0-\gap,0,0)--(0-\gap,0,1.25);
\draw [->] (0-\gap,0,1.25)--(0-\gap,0,4);
\draw [->] (0-\gap,1,0)--(0-\gap,3,0);
\draw [->] (1-\gap,0,0)--(3-\gap,0,0);

\node [below] at (3-\gap,0,0) {\tiny$t_1$};
\node [left] at (0-\gap,3,0) {\tiny$t_2$};
\node [below] at (0-\gap,0,4) {\tiny$t_3$};

\node [above right] at (0.8-\gap,0.9,0) {\tiny$(s,s,0)$};
\node [right] at (0.8-\gap,0,1.7) {\tiny$(s,0,s)$};
\node [ left] at (0.4-\gap,0,1.7) {\tiny$(0,0,s)$};
\node [left] at (0.4-\gap,1.3,1.7) {\tiny$(0,s,s)$};
\node [right] at (0.8-\gap,1.3,1.7) {\tiny$(s,s,s)$};
\node [below right] at (0.8-\gap,0,-0.2) {\tiny$(s,0,0)$};
\node [above left] at (0.8-\gap,0.9,0) {\tiny$(0,s,0)$};

\node  at (0.6-\gap,0.4,0.8) {$\mathcal{C}^3(s)$};

\node [below] at (1.2-\gap,-1.5,1.3) {in $\mathrm{ \bf Claim}~ X(s)$};

\path[fill=green!5] (0,0,3)--(1,0,3)--(1,1,3)--(0,1,3)--(0,0,3);
\path[fill=green!5] (2.25,0,1.25)--(2.25,0,0)--(2.25,1,0)--(2.25,1,1.25)--(2.25,0,1.25);
\path[fill=green!5] (1,2.25,0)--(0,2.25,0)--(0,2.25,1.25)--(1,2.25,1.25)--(1,2.25,0);

\draw[dotted] (1,1,3)--(1,2.25,1.25)--(2.25,1,1.25)--(1,1,3);
\draw[dotted] (1,1,3)--(1,0,3)--(2.25,0,1.25);
\draw[dotted] (1,2.25,0)--(2.25,1,0)--(2.25,1,1.25);
\draw[dotted] (0,1,3)--(0,2.25,1.25)--(1,2.25,1.25);

\draw[dash pattern= { on 2pt off 1pt}] (1,0,3)--(1,1,3)--(0,1,3);
\draw[dash pattern= { on 2pt off 1pt}]  (0,2.25,1.25)--(1,2.25,1.25)--(1,2.25,0);
\draw[dash pattern= { on 2pt off 1pt}] (2.25,0,1.25)--(2.25,1,1.25)--(2.25,1,0);

\draw[dash pattern= { on 2pt off 1pt}] (2.25,0,1.25)--(2.25,0,0)--(2.25,1,0);
\draw[dash pattern= { on 2pt off 1pt}] (0,2.25,1.25)--(0,2.25,0)--(1,2.25,0);
\draw[dash pattern= { on 2pt off 1pt}] (1,0,3)--(0,0,3)--(0,1,3);

\draw[dash pattern= { on 1pt off 1pt}] (0,0,0)--(0.3,0,0);
\draw[dash pattern= { on 1pt off 1pt}] (1.5,0,0)--(2.25,0,0);
\draw[dash pattern= { on 1pt off 1pt}] (0,0,0)--(0,0.3,0);
\draw[dash pattern= { on 1pt off 1pt}] (0,1.5,0)--(0,2.25,0);
\draw[dash pattern= { on 1pt off 1pt}] (0,0,0)--(0,0,3);
\draw [->] (0,0,3)--(0,0,4);
\draw [->] (0,2.25,0)--(0,3,0);
\draw [->] (2.25,0,0)--(3,0,0);

\draw [-] (0.3,0,0)--(1.5,0,0);
\draw [-] (0,0.3,0)--(0,1.5,0);

\node [below] at (3,0,0) {\tiny$t_1$};
\node [left] at (0,3,0) {\tiny$t_2$};
\node [below] at (0,0,4) {\tiny$t_3$};

\node  at (2.4,0.7,0.7) {$\mathscr{R}_1^3(s)$};
\node  at (0.5,2.25,0.7) {$\mathscr{R}^3_2(s)$};
\node  at (0.51,0.5,3) {$\mathscr{R}^3_3(s)$};

\node [left] at (0,2.25,-0.2) {\tiny$(0,1,0)$};
\node [right] at (0.9,2.25,-0.2) {\tiny$(s,1,0)$};
\node [left] at (0,2.25,1.1) {\tiny$(0,1,s)$};
\node [right] at (1,2.25,1.2) {\tiny$(s,1,s)$};

\node [right] at (2.1,1.2,0) {\tiny$(1,s,0)$};
\node [right] at (2.1,0.1,0) {\tiny$(1,0,0)$};
\node [right] at (2,0,1.6) {\tiny$(1,0,s)$};

\node [right] at (0.9,0,3) {\tiny$(s,0,1)$};
\node [left] at (0,0,2.8) {\tiny$(0,0,1)$};
\node [left] at (0,1,2.8) {\tiny$(0,s,1)$};

\node [below] at (1.2,-1.5,1.3) {in $\mathrm{ \bf Claim}~ Y(s)$};


\path[fill=green!5] (0-\gap,0-\gaptt,1.25)--(0-\gap,2.25-\gaptt,1.25)--(0-\gap,2.25-\gaptt,0)--(1-\gap,2.25-\gaptt,0)--(1-\gap,0-\gaptt,-0.2)--(1-\gap,0-\gaptt,1.25)--(0-\gap,0-\gaptt,1.25);
\path[fill=green!5] (0-\gap,1-\gaptt,0)--(2.25-\gap,1-\gaptt,0)--(2.25-\gap,0-\gaptt,0)--(2.25-\gap,0-\gaptt,1.25)--(0-\gap,0-\gaptt,1.25);
\path[fill=green!5] (0-\gap,1-\gaptt,0)--(0-\gap,1-\gaptt,3)--(0-\gap,0-\gaptt,3)--(1-\gap,0-\gaptt,3)--(1-\gap,0-\gaptt,0);

\draw[dotted](1-\gap,1-\gaptt,3)--(1-\gap,2.25-\gaptt,1.25)--(2.25-\gap,1-\gaptt,1.25)--(1-\gap,1-\gaptt,3);
\draw[dotted] (1-\gap,1-\gaptt,3)--(1-\gap,0-\gaptt,3)--(2.25-\gap,0-\gaptt,1.25);
\draw[dotted] (1-\gap,2.25-\gaptt,0)--(2.25-\gap,1-\gaptt,0);
\draw[dotted] (1-\gap,1-\gaptt,3)--(0-\gap,1-\gaptt,3)--(0-\gap,2.25-\gaptt,1.25);
\draw[dash pattern= { on 2pt off 1pt}] (2.25-\gap,0-\gaptt,1.25)--(2.25-\gap,0-\gaptt,0)--(2.25-\gap,1-\gaptt,0)--(2.25-\gap,1-\gaptt,1.25)--(2.25-\gap,0-\gaptt,1.25);
\draw[dash pattern= { on 2pt off 1pt}] (0-\gap,2.25-\gaptt,1.25)--(0-\gap,2.25-\gaptt,0)--(1-\gap,2.25-\gaptt,0)--(1-\gap,2.25-\gaptt,1.25)--(0-\gap,2.25-\gaptt,1.25);
\draw[dash pattern= { on 2pt off 1pt}] (1-\gap,0-\gaptt,3)--(0-\gap,0-\gaptt,3)--(0-\gap,1-\gaptt,3)--(1-\gap,1-\gaptt,3)--(1-\gap,0-\gaptt,3);

\draw[dash pattern= { on 1pt off 1pt}] (0-\gap,0-\gaptt,0)--(2.25-\gap,0-\gaptt,0);
\draw[dash pattern= { on 1pt off 1pt}] (0-\gap,0-\gaptt,0)--(0-\gap,2.25-\gaptt,0);
\draw[dash pattern= { on 1pt off 1pt}] (0-\gap,0-\gaptt,0)--(0-\gap,0-\gaptt,3);
\draw [->] (0-\gap,0-\gaptt,3)--(0-\gap,0-\gaptt,4);
\draw [->] (0-\gap,2.25-\gaptt,0)--(0-\gap,3-\gaptt,0);
\draw [->] (2.25-\gap,0-\gaptt,0)--(3-\gap,0-\gaptt,0);

\node [below] at (3-\gap,0-\gaptt,0) {\tiny$t_1$};
\node [left] at (0-\gap,3-\gaptt,0) {\tiny$t_2$};
\node [below] at (0-\gap,0-\gaptt,4) {\tiny$t_3$};

\node [left] at (0-\gap,2.25-\gaptt,-0.2) {\tiny$(0,1,0)$};
\node [right] at (0.9-\gap,2.25-\gaptt,-0.2) {\tiny$(s,1,0)$};
\node [left] at (0-\gap,2.25-\gaptt,1.1) {\tiny$(0,1,s)$};

\draw[dash pattern= { on 2pt off 1pt}] (0-\gap,1-\gaptt,0)--(2.25-\gap,1-\gaptt,0);
\draw[dash pattern= { on 2pt off 1pt}] (0-\gap,1-\gaptt,1.25)--(2.25-\gap,1-\gaptt,1.25);
\draw[dash pattern= { on 2pt off 1pt}] (0-\gap,0-\gaptt,1.25)--(2.25-\gap,0-\gaptt,1.25);

\draw[dash pattern= { on 2pt off 1pt}] (0-\gap,1-\gaptt,0)--(0-\gap,1-\gaptt,3);
\draw[dash pattern= { on 2pt off 1pt}] (1-\gap,1-\gaptt,0)--(1-\gap,1-\gaptt,3);
\draw[dash pattern= { on 2pt off 1pt}] (1-\gap,0-\gaptt,0)--(1-\gap,0-\gaptt,3);

\draw[dash pattern= { on 2pt off 1pt}] (1-\gap,0-\gaptt,0)--(1-\gap,2.25-\gaptt,0);
\draw[dash pattern= { on 2pt off 1pt}] (1-\gap,0-\gaptt,1.25)--(1-\gap,2.25-\gaptt,1.25);
\draw[dash pattern= { on 2pt off 1pt}] (0-\gap,0-\gaptt,1.25)--(0-\gap,2.25-\gaptt,1.25);

\node [right] at (2.1-\gap,1-\gaptt,-0.2) {\tiny$(1,s,0)$};
\node [right] at (2.1-\gap,0.1-\gaptt,-0.2) {\tiny$(1,0,0)$};
\node [right] at (2.25-\gap,0-\gaptt,1.25) {\tiny$(1,0,s)$};

\node [right] at (1-\gap,0-\gaptt,3) {\tiny$(s,0,1)$};
\node [left] at (0-\gap,0-\gaptt,2.8) {\tiny$(0,0,1)$};
\node [left] at (0-\gap,1-\gaptt,2.8) {\tiny$(0,s,1)$};

\node  at (2.1-\gap,0.7-\gaptt,0.9) {$\mathscr{V}_1^3(s)$};
\node  at (0.8-\gap,1.9-\gaptt,0.9) {$\mathscr{V}^3_2(s)$};
\node  at (0.8-\gap,0.7-\gaptt,2.8) {$\mathscr{V}^3_3(s)$};

\node [below] at (1.2-\gap,-1.5-\gaptt,1.3) {in $\mathrm{ \bf Claim}~ Z(s)$};


\path[fill=green!5] (0,0-\gaptt,3)--(1,0-\gaptt,3)--(2.25,0-\gaptt,1.25)--(2.25,0-\gaptt,0)--(2.25,1-\gaptt,0)--(1,2.25-\gaptt,0)--(0,2.25-\gaptt,0)--(0,2.25-\gaptt,1.25)--(0,1-\gaptt,3)--(0,0-\gaptt,3);

\draw[dash pattern= { on 2pt off 1pt}](1,1-\gaptt,3)--(1,2.25-\gaptt,1.25)--(2.25,1-\gaptt,1.25)--(1,1-\gaptt,3);
\draw[dash pattern= { on 2pt off 1pt}] (1,1-\gaptt,3)--(1,0-\gaptt,3)--(2.25,0-\gaptt,1.25)--(2.25,1-\gaptt,1.25);
\draw[dash pattern= { on 2pt off 1pt}] (1,2.25-\gaptt,1.25)--(1,2.25-\gaptt,0)--(2.25,1-\gaptt,0)--(2.25,1-\gaptt,1.25);
\draw[dash pattern= { on 2pt off 1pt}] (1,1-\gaptt,3)--(0,1-\gaptt,3)--(0,2.25-\gaptt,1.25)--(1,2.25-\gaptt,1.25);
\draw[dash pattern= { on 2pt off 1pt}] (2.25,0-\gaptt,1.25)--(2.25,0-\gaptt,0)--(2.25,1-\gaptt,0);
\draw[dash pattern= { on 2pt off 1pt}] (0,2.25-\gaptt,1.25)--(0,2.25-\gaptt,0)--(1,2.25-\gaptt,0);
\draw[dash pattern= { on 2pt off 1pt}] (1,0-\gaptt,3)--(0,0-\gaptt,3)--(0,1-\gaptt,3);

\draw[dash pattern= { on 1pt off 1pt}] (0,0-\gaptt,0)--(2.25,0-\gaptt,0);
\draw[dash pattern= { on 1pt off 1pt}] (0,0-\gaptt,0)--(0,2.25-\gaptt,0);
\draw[dash pattern= { on 1pt off 1pt}] (0,0-\gaptt,0)--(0,0-\gaptt,3);
\draw [->] (0,0-\gaptt,3)--(0,0-\gaptt,4);
\draw [->] (0,2.25-\gaptt,0)--(0,3-\gaptt,0);
\draw [->] (2.25,0-\gaptt,0)--(3,0-\gaptt,0);

\node [below] at (3,0-\gaptt,0) {\tiny$t_1$};
\node [left] at (0,3-\gaptt,0) {\tiny$t_2$};
\node [below] at (0,0-\gaptt,4) {\tiny$t_3$};

\node [left] at (0,2.25-\gaptt,0) {\tiny$(0,1,0)$};
\node [right] at (1,2.25-\gaptt,0) {\tiny$(s,1,0)$};
\node [left] at (0,2.25-\gaptt,1.1) {\tiny$(0,1,s)$};
\node [right] at (1,2.25-\gaptt,1.2) {\tiny$(s,1,s)$};

\node [right] at (2.15,1.1-\gaptt,0) {\tiny$(1,s,0)$};
\node [right] at (2.15,0.2-\gaptt,0) {\tiny$(1,0,0)$};
\node [right] at (2.1,-0.1-\gaptt,1.25) {\tiny$(1,0,s)$};

\node [right] at (1,0-\gaptt,3) {\tiny$(s,0,1)$};
\node [left] at (0.1,0-\gaptt,3) {\tiny$(0,0,1)$};
\node [left] at (0.1,1-\gaptt,3) {\tiny$(0,s,1)$};

\node  at (1.2,1-\gaptt,1.5) {$\HH^3(s)$};

\node [below] at (1.2,-1.5-\gaptt,1.3) {in $\mathrm{ \bf Claim}~ \Sigma(s)$};

\end{tikzpicture}
\caption{The trilinear case $m=3$ : the range of $(\frac{1}{p_1},\frac{1}{p_2},\frac{1}{p_3})$}\label{fig1}
\end{figure}
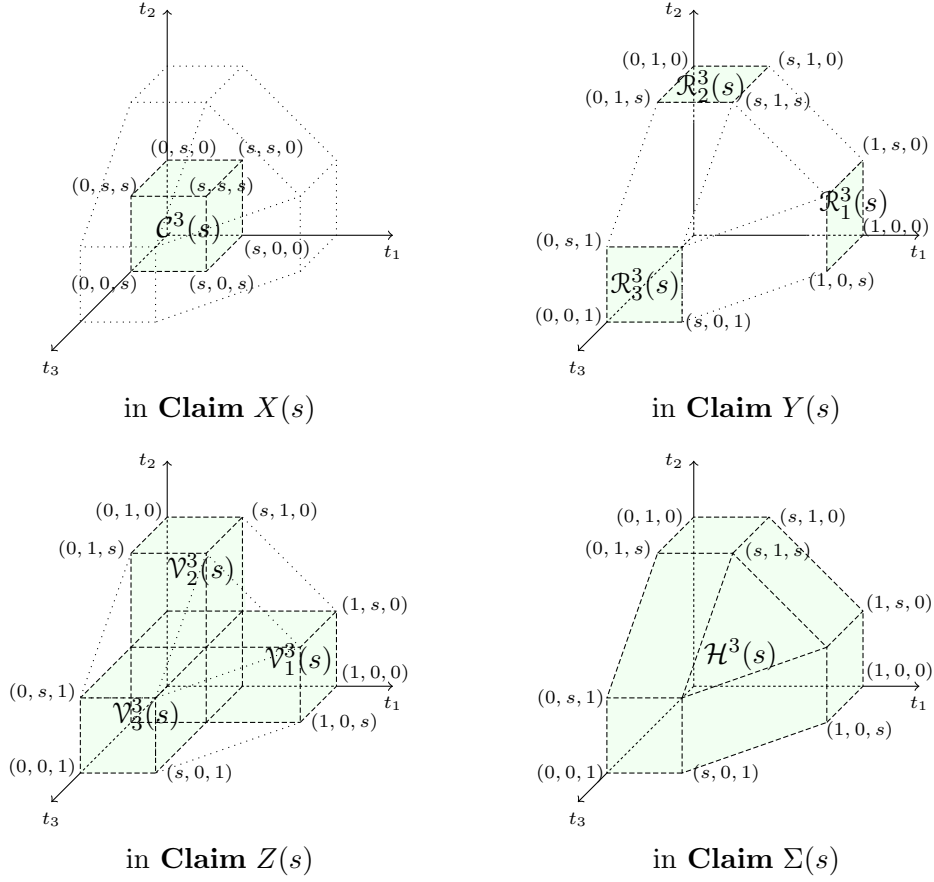
Then we will carry out induction arguments through the following proposition.
\begin{proposition}\label{inductionpropo}
Let $0< s<1$.
Then we have
$$ \mathrm{ \bf Claim}~ X(s) \Rightarrow  \mathrm{ \bf Claims}~X(s)\text{ and } Y(s) \Rightarrow \mathrm{ \bf Claim}~ Z(s) \Rightarrow \mathrm{ \bf Claim}~ \Sigma(s).$$
\end{proposition}

\hfill

Let us temporarily take Proposition \ref{inductionpropo} for granted and complete the proof of Proposition \ref{hsmainestpropo}.

We first consider the case $0<s<\frac{1}{m}$. 
In this case,  if $(\frac{1}{p_1},\cdots,\frac{1}{p_m})\in \mathcal{C}^m(s)$,
then Proposition \ref{keyproposharp} yields
\begin{equation*}
\big\| \LL_{\Om,\mu}^{\sharp}(f_1,\dots,f_m )\big\|_{L^{p}(\bbrn)}\lesssim_M \mu^{M} \|\Om\|_{L^1(\mathbb S^{mn-1})}
\prod_{j=1}^{m}\|f_j\|_{L^{p_j}(\bbrn)}.
\end{equation*}
Since  $$\Vert \Omega\Vert_{L^1(\mathbb{S}^{mn-1})}\lesssim \Vert \Omega\Vert_{L^{\frac{1}{1-s}}(\mathbb{S}^{mn-1})}$$
 and  for any $\epsilon>0$
$$\mu^M\lesssim_{\epsilon,M}2^{\epsilon \mu}, \qq \mu\in\bbn,$$
 $\mathrm{\bf Claim}~X(s)$ holds.
Then Proposition \ref{inductionpropo} deduces \eqref{keypropoesttri}, as desired.

Now assume $\frac{1}{m}\le s< 1$.
For $\nu\in \bbn$, let
$$a_{\nu}:=1-\Big(1-\frac{1}{m} \Big)^{\nu}.$$
Then we observe that $(a_{\nu+1},\dots,a_{\nu+1})\in \bbr^m$ is the center of the $(m-1)$ simplex with $m$ vertices 
$(1,a_{\nu},a_{\nu},\dots,a_{\nu})$, $(a_{\nu},1,a_{\nu},\dots, a_{\nu})$, $\dots$, $(a_{\nu},\dots, a_{\nu},1,a_{\nu})$, and $(a_{\nu},\dots, a_{\nu},a_{\nu},1)$.
The trilinear case ($m=3$) is illustrated in Figure \ref{fig0}.     
 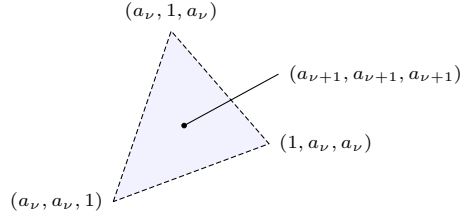
\begin{figure}[h]
\begin{tikzpicture}

\path[fill=blue!5] (1.7,1.5,4)--(1.7,3,2)--(3,1.5,2)--(1.7,1.5,4);

\draw[dash pattern= { on 2pt off 1pt}] (1.7,1.5,4)--(1.7,3,2)--(3,1.5,2)--(1.7,1.5,4);

\filldraw[fill=black] (2.13,2,2.67)  circle[radius=0.3mm];
\draw [-] (2.13,2,2.67)--(3,2.3,1.67);
\node [right] at (3,2.3,1.67) {\tiny$(a_{\nu+1},a_{\nu+1},a_{\nu+1})$};

\node [right] at (3,1.5,2) {\tiny$(1,a_{\nu},a_{\nu})$};
\node [above] at (1.7,3,2) {\tiny$(a_{\nu},1,a_{\nu})$};
\node [left] at (1.7,1.5,4) {\tiny$(a_{\nu},a_{\nu},1)$};

\end{tikzpicture}
\caption{$(a_{\nu+1},a_{\nu+1},a_{\nu+1})$ when $m=3$}\label{fig0}
\end{figure}
We notice that $a_1=\frac{1}{m}$, 
$a_{\nu+1}=\frac{a_{\nu}(m-1)+1}{m}$ for $\nu\ge 1$,   and $a_{\nu}\nearrow 1$ as $\nu \to \infty$. 
Moreover,  by the definition of $\HH^m(a_{\nu})$  we have
$$\mathcal{C}^m(a_{\nu+1})\subset \HH^m(a_{\nu}) \q \text{ for all }~\nu\in \bbn,$$
 see Figure \ref{fig5},
which implies
\begin{equation}\label{zanutoxanu1}
 \mathrm{ \bf Claim}~ \Sigma(a_{\nu})\Rightarrow \mathrm{ \bf Claim}~ X(a_{\nu+1})  \q \text{ for all }~\nu\in \bbn
 \end{equation}
 as $L^{\frac{1}{1-a_{\nu+1}}}(\mathbb{S}^{mn-1})\hookrightarrow L^{\frac{1}{1-a_{\nu}}}(\mathbb{S}^{mn-1})$.
  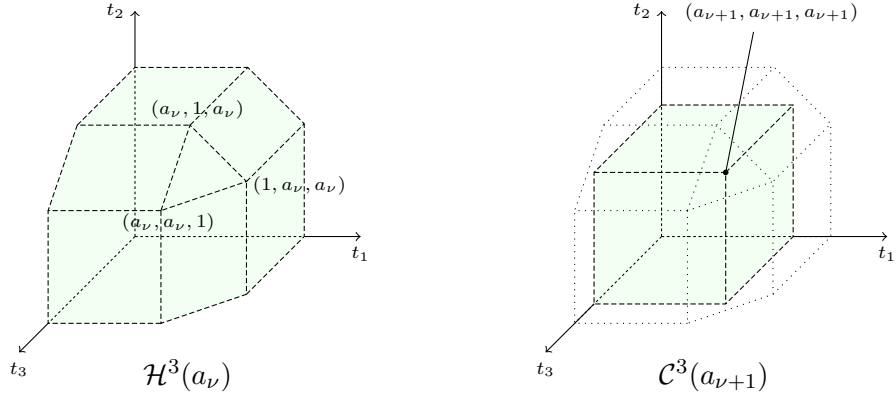
\begin{figure}[h]
\begin{tikzpicture}

\path[fill=green!5] (0+\gapt,0,3)--(1.5+\gapt,0,3)--(2.25+\gapt,0,1.995)--(2.25+\gapt,0,0)--(2.25+\gapt,1.5,0)--(1.5+\gapt,2.25,0)--(0+\gapt,2.25,0)--(0+\gapt,2.25,1.995)--(0+\gapt,1.5,3)--(0+\gapt,0,3);

\draw[dash pattern= { on 2pt off 1pt}](1.5+\gapt,1.5,3)--(1.5+\gapt,2.25,1.995)--(2.25+\gapt,1.5,1.995)--(1.5+\gapt,1.5,3);
\draw[dash pattern= { on 2pt off 1pt}] (1.5+\gapt,1.5,3)--(1.5+\gapt,0,3)--(2.25+\gapt,0,1.995)--(2.25+\gapt,1.5,1.995);
\draw[dash pattern= { on 2pt off 1pt}] (1.5+\gapt,2.25,1.995)--(1.5+\gapt,2.25,0)--(2.25+\gapt,1.5,0)--(2.25+\gapt,1.5,1.995);
\draw[dash pattern= { on 2pt off 1pt}] (1.5+\gapt,1.5,3)--(0+\gapt,1.5,3)--(0+\gapt,2.25,1.995)--(1.5+\gapt,2.25,1.995);
\draw[dash pattern= { on 2pt off 1pt}] (2.25+\gapt,0,1.995)--(2.25+\gapt,0,0)--(2.25+\gapt,1.5,0);
\draw[dash pattern= { on 2pt off 1pt}] (0+\gapt,2.25,1.995)--(0+\gapt,2.25,0)--(1.5+\gapt,2.25,0);
\draw[dash pattern= { on 2pt off 1pt}] (1.5+\gapt,0,3)--(0+\gapt,0,3)--(0+\gapt,1.5,3);

\node [right] at (2.2+\gapt,1.45,1.995) {\tiny$(1,a_{\nu},a_{\nu})$};
\node [above] at (1.6+\gapt,2.2,2) {\tiny$(a_{\nu},1,a_{\nu})$};
\node [below] at (1.6+\gapt,1.6,3) {\tiny$(a_{\nu},a_{\nu},1)$};

\draw[dash pattern= { on 1pt off 1pt}] (0+\gapt,0,0)--(2.25+\gapt,0,0);
\draw[dash pattern= { on 1pt off 1pt}] (0+\gapt,0,0)--(0+\gapt,2.25,0);
\draw[dash pattern= { on 1pt off 1pt}] (0+\gapt,0,0)--(0+\gapt,0,3);
\draw [->] (0+\gapt,0,3)--(0+\gapt,0,4);
\draw [->] (0+\gapt,2.25,0)--(0+\gapt,3,0);
\draw [->] (2.25+\gapt,0,0)--(3+\gapt,0,0);

\node [below] at (3+\gapt,0,0) {\tiny$t_1$};
\node [left] at (0+\gapt,3,0) {\tiny$t_2$};
\node [below] at (0+\gapt,0,4) {\tiny$t_3$};

\node [below] at (1.2+\gapt,-1,1.3) {$\HH^3(a_{\nu})$};


\path[fill=green!5] (0+\gapt+\gapt,0,2.33)--(1.75+\gapt+\gapt,0,2.33)--(1.75+\gapt+\gapt,1.75,2.33)--(0+\gapt+\gapt,1.75,2.33)--(0+\gapt+\gapt,0,2.33);
\path[fill=green!5] (1.75+\gapt+\gapt,0,2.33)--(1.75+\gapt+\gapt,0,0)--(1.75+\gapt+\gapt,1.75,0)--(1.75+\gapt+\gapt,1.75,2.33)--(1.75+\gapt+\gapt,0,2.33);
\path[fill=green!5] (1.75+\gapt+\gapt,1.75,0)--(0+\gapt+\gapt,1.75,0)--(0+\gapt+\gapt,1.75,2.33)--(1.75+\gapt+\gapt,1.75,2.33)--(1.75+\gapt+\gapt,1.75,0);

\draw[dash pattern= { on 2pt off 1pt}] (0+\gapt+\gapt,1.75,0)--(1.75+\gapt+\gapt,1.75,0)--(1.75+\gapt+\gapt,0,0)--(1.75+\gapt+\gapt,0,2.33)--(0+\gapt+\gapt,0,2.33)--(0+\gapt+\gapt,1.75,2.33)--(0+\gapt+\gapt,1.75,0);
\draw[dash pattern= { on 2pt off 1pt}] (0+\gapt+\gapt,1.75,2.33)--(1.75+\gapt+\gapt,1.75,2.33)--(1.75+\gapt+\gapt,0,2.33);
\draw[dash pattern= { on 2pt off 1pt}] (1.75+\gapt+\gapt,1.75,2.33)--(1.75+\gapt+\gapt,1.75,0);

\draw[dotted] (1.5+\gapt+\gapt,1.5,3)--(1.5+\gapt+\gapt,2.25,1.995)--(2.25+\gapt+\gapt,1.5,1.995)--(1.5+\gapt+\gapt,1.5,3);
\draw[dotted] (1.5+\gapt+\gapt,1.5,3)--(1.5+\gapt+\gapt,0,3)--(2.25+\gapt+\gapt,0,1.995)--(2.25+\gapt+\gapt,1.5,1.995);
\draw[dotted] (1.5+\gapt+\gapt,2.25,1.995)--(1.5+\gapt+\gapt,2.25,0)--(2.25+\gapt+\gapt,1.5,0)--(2.25+\gapt+\gapt,1.5,1.995);
\draw[dotted](1.5+\gapt+\gapt,1.5,3)--(0+\gapt+\gapt,1.5,3)--(0+\gapt+\gapt,2.25,1.995)--(1.5+\gapt+\gapt,2.25,1.995);
\draw[dotted] (2.25+\gapt+\gapt,0,1.995)--(2.25+\gapt+\gapt,0,0)--(2.25+\gapt+\gapt,1.5,0);
\draw[dotted] (0+\gapt+\gapt,2.25,1.995)--(0+\gapt+\gapt,2.25,0)--(1.5+\gapt+\gapt,2.25,0);
\draw[dotted] (1.5+\gapt+\gapt,0,3)--(0+\gapt+\gapt,0,3)--(0+\gapt+\gapt,1.5,3);

\draw[dash pattern= { on 1pt off 1pt}] (0+\gapt+\gapt,0,0)--(2.25+\gapt+\gapt,0,0);
\draw[dash pattern= { on 1pt off 1pt}] (0+\gapt+\gapt,0,0)--(0+\gapt+\gapt,2.25,0);
\draw[dash pattern= { on 1pt off 1pt}] (0+\gapt+\gapt,0,0)--(0+\gapt+\gapt,0,3);
\draw [->] (0+\gapt+\gapt,0,2.33)--(0+\gapt+\gapt,0,4);
\draw [->] (0+\gapt+\gapt,1.75,0)--(0+\gapt+\gapt,3,0);
\draw [->] (1.75+\gapt+\gapt,0,0)--(3+\gapt+\gapt,0,0);

\node [below] at (3+\gapt+\gapt,0,0) {\tiny$t_1$};
\node [left] at (0+\gapt+\gapt,3,0) {\tiny$t_2$};
\node [below] at (0+\gapt+\gapt,0,4) {\tiny$t_3$};

\filldraw[fill=black] (1.75+\gapt+\gapt,1.75,2.33)  circle[radius=0.3mm];
\draw [-] (1.75+\gapt+\gapt,1.75,2.33)--(1.05+\gapt+\gapt,2.55,-0.45);
\node [above right] at (0+\gapt+\gapt,2.55,-0.45) {\tiny$(a_{\nu+1},a_{\nu+1},a_{\nu+1})$};

\node [below] at (1.2+\gapt+\gapt,-1,1.3) {$\mathcal{C}^3(a_{\nu+1})$};

\end{tikzpicture}
\caption{The trilinear case $m=3$ : $\HH^3(a_{\nu})$ and $\mathcal{C}^3(a_{\nu+1})$}\label{fig5}
\end{figure}
Now $\mathrm{\bf Claim} ~X(a_1)$ holds due to Proposition \ref{keyproposharp}, and accordingly, $\mathrm{\bf Claim} ~\Sigma(a_{\nu})$ should be also true for all $\nu\in\bbn$ with the aid of Proposition \ref{inductionpropo} and (\ref{zanutoxanu1}).
When $s=\frac{1}{m}~(=a_1)$, the asserted estimate \eqref{keypropoesttri} is exactly $\mathrm{ \bf Claim}~ \Sigma(a_{1})$.
If $a_{\nu}<s\le a_{\nu+1}$ for some $\nu\in\bbn$, then $\mathcal{C}^m(s)\subset \mathcal{H}^m(a_{\nu})$, and this yields that
$ \mathrm{ \bf Claim}~ X(s)$ holds since $L^{\frac{1}{1-s}}(\mathbb{S}^{mn-1})\hookrightarrow L^{\frac{1}{1-a_{\nu}}}(\mathbb{S}^{mn-1})$.
Finally, Proposition \ref{inductionpropo} shows that $ \mathrm{ \bf Claim}~ \Sigma (s)$ works. This finishes the proof of Proposition \ref{hsmainestpropo}.\\

In the rest of this section, we will prove Proposition \ref{inductionpropo}.
\begin{proof}[Proof of Proposition \ref{inductionpropo}]
Let $0<s<1$.
We first note that 
the direction $$\mathrm{ \bf Claims}~ X(s) \text{ and } Y(s)\Rightarrow \mathrm{\bf Claim}~Z(s)$$
follows from the (sublinear) Marcinkiewicz interpolation method. Here, we apply the interpolation separately $m$ times and in each interpolation, $m-1$ parameters among $p_1,\dots,p_m$ are fixed.
Moreover, the direction 
$$\mathrm{\bf Claim}~Z(s) \Rightarrow \mathrm{\bf Claim}~\Sigma(s)$$
also holds due to  Lemma \ref{interpollemma} and the geometric property \eqref{convexhull}.
Therefore it remains to show the direction
$\mathrm{ \bf Claim}~ X(s) \Rightarrow  \mathrm{ \bf Claim}~ Y(s)$.
For this one, we deal with only the case $(\frac{1}{p_1},\dots,\frac{1}{p_m})\in\mathscr{R}_1^m(s)$, appealing to symmetry for other cases.
 Assume that $p_1=1$, $\frac{1}{s}<p_2,\dots,p_m< \infty$, and 
 \begin{equation*}
 1+\frac{1}{p_2}+\cdots+\frac{1}{p_m}=\frac{1}{p}.
 \end{equation*}
Without loss of generality, we may also assume 
$$\Vert f_1\Vert_{L^1(\bbrn)}=\Vert f_2\Vert_{L^{p_2}(\bbrn)}=\cdots=\Vert f_m\Vert_{L^{p_m}(\bbrn)}=\Vert \Omega\Vert_{L^{\frac{1}{1-s}}(\mathbb{S}^{mn-1})}=1$$ and then it suffices to prove that for any $\epsilon>0$
\begin{equation}\label{weakmainest}
\Big|\Big\{    x\in \bbrn:     \LL_{\Om,\mu}^{\sharp}(f_1,\dots,f_m)(x) >\lambda    \Big\}\Big|\lesssim_{\epsilon} 2^{\epsilon \mu  p }\frac{1}{\lambda^p}.
\end{equation}
 Using the Calder\'on-Zygmund decomposition of $f_1$ at height $\lambda^p$, we write $f_1$  as
 $$f_1=g_1+\sum_{Q\in\mathcal{A}}{b_{1,Q}}$$
 where $\mathcal{A}$ is a subset of disjoint dyadic cubes, $\big| \bigcup_{Q\in \mathcal{A}}Q\big|\lesssim \frac{1}{\lambda^p}$, $\supp(b_{1,Q})\subset Q$, $\int{b_{1,Q}(y)}dy=0$, $\Vert b_{1,Q}\Vert_{L^1(\bbrn)}\lesssim \lambda^{p}|Q|$, and $\Vert g_1\Vert_{L^r(\bbrn)}\lesssim \lambda^{(1-\frac{1}{r})p}$ for all $1\le r\le \infty$.
Then the left-hand side of (\ref{weakmainest}) is controlled by the sum of 
\begin{equation*}
\Xi_1^{\mu}:=\Big|\Big\{    x\in \bbrn:    \big|\LL_{\Om,\mu}^{\sharp}(g_1,f_2,\dots,f_m)(x)\big|>\frac{\lambda}{2}    \Big\}\Big|
\end{equation*}
and
\begin{equation*}
\Xi_2^{\mu}:=\bigg|\bigg\{    x\in \bbrn:    \Big|\LL_{\Om,\mu}^{\sharp}\Big(\sum_{Q\in\mathcal{A}}{b_{1,Q}},f_2,\dots,f_m\Big)(x)\Big|>\frac{\lambda}{2}    \bigg\}\bigg|.
\end{equation*}

In order to estimate $\Xi_1^{\mu}$, we choose $\frac{1}{s}<p_0<\infty$ and $\wt{p}>p$ satisfying
\begin{equation*}
\frac{1}{p_0}+\frac{1}{p_2}+\dots+\frac{1}{p_m}=\frac{1}{\wt{p}}
\end{equation*}
and set $\epsilon_0:=\frac{\epsilon p}{\wt{p}}$ so that $0<\epsilon_0<1$.
Then it follows from the hypothesis $\mathrm{ \bf Claim}~ X(s)$ that
\begin{equation}\label{claimxnu}
\big\Vert \LL_{\Om,\mu}^{\sharp}(g_1,f_2,\dots,f_m)\big\Vert_{L^{\wt{p}}(\bbrn)}\lesssim_{\epsilon_0}2^{\epsilon_0 \mu} \Vert g_1\Vert_{L^{p_0}(\bbrn)}\lesssim 2^{\epsilon_0\mu}\lambda^{(1-\frac{1}{p_0})p}.
\end{equation}
Now, Chebyshev's inequality and the estimate (\ref{claimxnu}) yield
\begin{align*}
\Xi_1^{\mu}\lesssim \frac{1}{\lambda^{\wt{p}}} \big\Vert \LL_{\Om,\mu}^{\sharp}(g_1,f_2,\dots,f_m)\big\Vert_{L^{\wt{p}}(\bbrn)}^{\wt{p}}\lesssim 2^{\epsilon_0 \mu \wt{p}}\lambda^{\wt{p}( (1-\frac{1}{p_0})p-1)}=2^{\epsilon \mu p}\frac{1}{\lambda^{p}},
\end{align*}
as desired. Here, we note that $\frac{1}{\wt{p}}-\frac{1}{p_0}=\frac{1}{p}-1$, which implies $\wt{p}( (1-\frac{1}{p_0})p-1)=-p$.

On the other hand, the term $\Xi_2^{\mu}$ is bounded by the sum of $\big| \bigcup_{Q\in\mathcal{A}}Q^*\big|$ and
 \begin{equation*}
\Gamma_{\mu}:= \bigg| \bigg\{x\in \Big( \bigcup_{Q\in\mathcal{A}}Q^*\Big)^c : \Big|  \LL_{\Om,\mu}^{\sharp}\Big(\sum_{Q\in\mathcal{A}}b_{1,Q},f_2,\dots,f_m\Big)(x)\Big|>\frac{\lambda}{2} \bigg\}\bigg|
 \end{equation*}
 where $Q^*$ is the concentric dilate of $Q$ with $\ell(Q^*)=10^2\sqrt{n}\ell(Q)$.
 Since  $\big| \bigcup_{Q\in\mathcal{A}}Q^*\big|\lesssim \frac{1}{\lambda^{p}}$, the estimate of $\Xi_2^{\mu}$ can be reduced to the inequality
 \begin{equation*}
 \Gamma_{\mu} \lesssim_{\epsilon}2^{\epsilon \mu p}\frac{1}{\lambda^{p}}.
 \end{equation*}
 Indeed, by applying Chebyshev's inequality, we obtain
 \begin{align*}
 \Gamma_{\mu}&\lesssim \frac{1}{\lambda^p}\int_{(\bigcup_{Q\in\mathcal{A}}Q^*)^c}  \sup_{\tau\in\bbz}\bigg| \sum_{\gamma<\tau} \sum_{Q\in\mathcal{A}}T_{K_{\mu}^{\gamma}}\big(b_{1,Q},f_2,\dots,f_m \big)(x) \bigg|^p               \; dx\\
 &\le  \frac{1}{\lambda^p}\int_{(\bigcup_{Q\in\mathcal{A}}Q^*)^c}\bigg( \sum_{Q\in\mathcal{A}}\sum_{\gamma\in\bbz} \big|T_{K_{\mu}^{\ga}}\big(b_{1,Q},f_2,\dots,f_m \big)(x) \big|\bigg)^{p} \; dx.
 \end{align*}
 Then it is already proved in \cite[(6.16)]{Gr_He_Ho_Park_JLMS} that the last expression is bounded by a constant times
 $$\frac{1}{\lambda^p}2^{\epsilon \mu p},$$
 which completes the proof of \eqref{weakmainest}.
\end{proof}

\hfill


\appendix

\section{Proof of Theorem \ref{maincor}}

Assume that $1<p_1,\dots,p_m<\infty$, $f_j\in L^{p_j}(\bbrn)$, $j=1,\dots,m$, and $\Omega\in L^q(\bbrn)$ for $1<q<\infty$ satisfying \eqref{qcondequi}, which clearly implies $\frac{1}{p_1}+\cdots+\frac{1}{p_m}=\frac{1}{p}<\frac{1}{q}+\frac{m}{q'}$.
According to Lemma \ref{maxomeest}, there exists a measure zero set $E^{\Omega}_{f_1,\dots,f_m}$ such that
\begin{equation}\label{momptest}
\mathcal{M}_{\Omega}\big(f_1,\dots,f_m\big)(x)<\infty, \qq x\in \bbrn\setminus E^{\Omega}_{f_1,\dots,f_m}.
\end{equation}
Since 
\begin{equation}\label{abininest}
 \int_{\ep_0 \le |\yyy| \le \ep_0^{-1}} \frac{|\Om (\yyy')|}{|\yyy|^{mn}}\,  \prod_{j=1}^{m}\big|f_j(x-y_j)\big|\; d\yyy \lesssim
\f{1 }{(\ep_0)^{2mn}} \mathcal M_{\Om} \big(f_1,\dots , f_m\big)(x), \q 0<\epsilon_0<1,    
\end{equation}
\eqref{momptest} yields $$\LL_{\Omega}^{*,\epsilon_0}\big(f_1,\dots,f_m\big)(x):=\sup_{\epsilon\ge \epsilon_0}\big|\LL_{\Omega}^{(\epsilon,\epsilon^{-1})}(f_1,\dots,f_m)(x) \big|$$
is finite for $x\in \bbrn\setminus E_{f_1,\dots,f_m}^{\Omega}$.
Obviously, $\LL_{\Omega}^{(\epsilon,\epsilon^{-1})}(f_1,\dots,f_m)$ is also well-defined on $ \bbrn\setminus E^{\Omega}_{f_1,\dots,f_m}$. 
For each $j=1,\dots,m$, we choose sequences $\{f_j^k\}_{k\in\bbn}$ of Schwartz functions such that $f_j^k$ converges to $f_j$ in $L^{p_j}(\bbrn)$ as $k\to \infty$.
Then applying Lemma \ref{maxomeest} many times, we may choose measure zero sets $E^{\Omega}_{f_{1}^k,\dots,f_m^k}$  on which 
$\mathcal{M}_{\Omega}\big(f_1^k,\dots,f_m^k\big)(x)$ is finite, and $E^{\Omega}_{f_1^k,\dots,f_{j-1}^k,f_j-f_j^k,f_{j+1},\dots,f_m}$ on which
$\mathcal{M}_{\Omega}\big(f_1^k,\dots,f_{j-1}^k,f_j-f_j^k,f_{j+1},\dots,f_m\big)(x)$ is finite.
Then, using \eqref{abininest}, we have
\begin{align*}
&\LL_{\Omega}^{*,\epsilon_0}(f_1,\dots,f_m)(x)\\
&\le 2\LL_{\Omega}^*\big(f_1^k,\dots,f_m^k\big)(x)+\sum_{j=1}^{m}\LL_{\Omega}^{*,\epsilon_0}\big(   f_1^k,\dots,f_{j-1}^k,f_j-f_j^k,f_{j+1},\dots,f_m     \big)(x)\\
&\lesssim \LL_{\Omega}^*\big(f_1^k,\dots,f_m^k\big)(x)+\frac{1}{(\epsilon_0)^{2mn}}\sum_{j=1}^{m}\mathcal{M}_{\Omega}\big(   f_1^k,\dots,f_{j-1}^k,f_j-f_j^k,f_{j+1},\dots,f_m     \big)(x)
\end{align*}
(with the usual modification when $j=1$ or $j=m$)
for any $0<\epsilon_0<1$ and  $x\in \bbrn\setminus E^{\Omega}$ , where 
\begin{equation}\label{eomegadef}
E^{\Omega}:=E^{\Omega}_{f_1,\dots,f_m}\cup \bigg( \bigcup_{k=1}^{\infty}E_{f_1^k,\dots,f_m^k}^{\Omega}\bigg)\cup\bigg( \bigcup_{j=1}^{m}\bigcup_{k=1}^{\infty}E_{f_1^k,\dots,f_{j-1}^k,f_j-f_j^k,f_{j+1},\dots,f_m}^{\Omega}  \bigg)
\end{equation}
which is also a set of measure zero.
Taking the $L^p$ (quasi-)norm on both sides and applying Theorem \ref{maintheorem} for the first term and Lemma \ref{maxomeest} for the other terms, it follows that
\begin{align*}
&\big\Vert \LL_{\Omega}^{*,\epsilon_0}(f_1,\dots,f_m)\big\Vert_{L^p(\bbrn)}\lesssim \Vert \Omega\Vert_{L^q(\mathbb{S}^{mn-1})}\prod_{j=1}^{m}\Vert f_j^k\Vert_{L^{p_j}(\bbrn)}\\
&\q+\frac{\Vert \Omega\Vert_{L^q(\mathbb{S}^{mn-1})}}{(\epsilon_0)^{2mn}}\sum_{j=1}^{m}\Big(\prod_{i=1}^{j-1}\Vert f_i^k\Vert_{L^{p_i}(\bbrn)}\Big) \big\Vert f_j-f_j^k\big\Vert_{L^{p_j}(\bbrn)}\Big(\prod_{i=j+1}^{m} \Vert f_{i}\Vert_{L^{p_{i}}(\bbrn)}\Big)
\end{align*}
and then the second parts vanishes as $k\to \infty$. Consequently, we have
\begin{equation}\label{lstarep0}
\big\Vert \LL_{\Omega}^{*,\epsilon_0}(f_1,\dots,f_m)\big\Vert_{L^p(\bbrn)}\lesssim \Vert \Omega\Vert_{L^q(\mathbb{S}^{mn-1})}\prod_{j=1}^{m}\Vert f_j\Vert_{L^{p_j}(\bbrn)}.
\end{equation}
We now define
$$\LL_{\Omega}^{**}(f_1,\dots,f_m):=\sup_{\epsilon>0}\big| \LL_{\Omega}^{(\epsilon,\epsilon^{-1})}(f_1,\dots,f_m)\big|=\lim_{\epsilon_0\searrow 0}\LL_{\Omega}^{*,\epsilon_0}(f_1,\dots,f_m),$$
which may be infinite. 
Then applying Fatou's lemma to \eqref{lstarep0}, we conclude
\begin{equation}\label{lpjftest}
\big\Vert \LL_{\Omega}^{**}\big(f_1,\dots,f_m\big)\big\Vert_{L^{p}(\bbrn)}\lesssim \Vert \Omega\Vert_{L^q(\mathbb{S}^{mn-1})}\prod_{j=1}^{m}\Vert f_j\Vert_{L^{p_j}(\bbrn)}
\end{equation}
when each $f_j$ belongs to $L^{p_j}(\bbrn)$.

Now let us finish the proof of Theorem \ref{maincor}.
Due to Theorem \ref{knownbdresult}, $\LL_{\Omega}(f_1,\dots,f_m)$ is defined as the $L^{p}$ limit of $\LL_{\Omega}(f_1^k,\dots,f_m^k)$ as $k\to \infty$.
Therefore, we may select a subsequence $\{k_l\}_{l\in\bbn}$ of $\{k\}_{k\in\bbn}$ so that $\LL_{\Omega}(f_1^{k_l},\cdots,f_m^{k_l})\to \LL_{\Omega}(f_1,\dots,f_m)$ pointwise on $\bbrn\setminus \mathscr{E}$ as $l\to \infty$ for some measure zero set $\mathscr{E}$ in $\bbrn$.
Then, setting $E^{\Omega}$ as in \eqref{eomegadef}, for $x\in \bbrn\setminus (E^{\Omega}\cup \mathscr{E})$,
\begin{align*}
&\big|\LL_{\Om}^{(\ep,\ep^{-1})}\big(f_1,\dots, f_m\big)(x) -\LL_{\Om}\big(f_1 ,\dots, f_m \big)(x)\big|\\
&\qq\le  \big|\LL_{\Om}^{(\ep,\ep^{-1})}\big(f_1,\dots, f_m\big)(x) -\LL_{\Om}^{(\ep,\ep^{-1})}\big(f_1^{k_l} ,\dots, f_m^{k_l} \big)(x)\big|\\
& \qq \qq+\big| \LL_{\Om}^{(\ep,\ep^{-1})}\big(f_1^{k_l} ,\dots, f_m^{k_l} \big)(x) - \LL_{\Om} \big(f_1^{k_l} ,\dots, f_m^{k_l} \big)(x) \big|\\
& \qq \qq\qq + \big| \LL_{\Om} \big(f_1^{k_l} ,\dots, f_m^{k_l} \big)(x) -\LL_{\Om}\big(f_1 ,\dots, f_m \big)(x) \big|.
\end{align*}
We first take the $\limsup_{\epsilon\searrow 0}$ on both sides to make the middle term on the right disappear. Then we apply $\liminf_{l\to \infty}$ so that the last term also vanishes. As a consequence, we have
\begin{align*}
&\limsup_{\epsilon\searrow 0}\big|\LL_{\Om}^{(\ep,\ep^{-1})}\big(f_1,\dots, f_m\big)(x) -\LL_{\Om}\big(f_1 ,\dots, f_m \big)(x)\big|\\
&\le  \liminf_{l\to\infty}\,\limsup_{\epsilon\searrow 0}\big|\LL_{\Om}^{(\ep,\ep^{-1})}\big(f_1,\dots, f_m\big)(x) -\LL_{\Om}^{(\ep,\ep^{-1})}\big(f_1^{k_l} ,\dots, f_m^{k_l} \big)(x)\big|\\
&\le \liminf_{l\to\infty}\,\limsup_{\epsilon\searrow 0}\sum_{j=1}^{m}\big| \LL_{\Omega}^{(\epsilon,\epsilon^{-1})}\big(f_1^{k_l},\dots,f_{j-1}^{k_l},f_j-f_j^{k_l},f_{j+1},\dots,f_m\big)(x)\big|\\
&\le  \liminf_{l\to\infty}\sum_{j=1}^{m}\mathcal{L}_{\Omega}^{**}\big(f_1^{k_l},\dots,f_{j-1}^{k_l},f_j-f_j^{k_l},f_{j+1},\dots,f_m\big)(x)
\end{align*}
for $x\in \bbrn\setminus (E^{\Omega}\cup \mathscr{E})$.
Since $E^{\Omega}\cup \mathscr{E}$ has measure zero, for any $\lambda>0$
\begin{align}\label{lastweak11est}
&\Big| \big\{x\in\bbrn: \limsup_{\epsilon\searrow 0}\big|\LL_{\Om}^{(\ep,\ep^{-1})}\big(f_1,\dots, f_m\big)(x) -\LL_{\Om}\big(f_1 ,\dots, f_m \big)(x)\big|>\lambda \big\}\Big|\nonumber\\
& \le \Big|\big\{x\in\bbrn:   \liminf_{l\to\infty}\sum_{j=1}^{m}\mathcal{L}_{\Omega}^{**}\big(f_1^{k_l},\dots,f_{j-1}^{k_l},f_j-f_j^{k_l},f_{j+1},\dots,f_m\big)(x)>\lambda \big\} \Big|\nonumber\\
&\le \frac{1}{\lambda^p}\bigg\Vert   \liminf_{l\to\infty}\sum_{j=1}^{m}\mathcal{L}_{\Omega}^{**}\big(f_1^{k_l},\dots,f_{j-1}^{k_l},f_j-f_j^{k_l},f_{j+1},\dots,f_m\big)     \bigg\Vert_{L^p(\bbrn)}^{p}\nonumber\\
&\lesssim \frac{1}{\lambda^p}\liminf_{l\to\infty} \sum_{j=1}^{m}\big\Vert    \mathcal{L}_{\Omega}^{**}\big(f_1^{k_l},\dots,f_{j-1}^{k_l},f_j-f_j^{k_l},f_{j+1},\dots,f_m\big)     \big\Vert_{L^p(\bbrn)}^{p}
\end{align}
where we applied Chebyshev's inequality and Fatou's lemma.
Applying \eqref{lpjftest} to 
$$\big(f_1^{k_l},\dots,f_{j-1}^{k_l},f_j-f_j^{k_l},f_{j+1},\dots,f_m\big)\in L^{p_1}(\bbrn)\times \cdots\times L^{p_m}(\bbrn), $$
we bound the right-hand side of \eqref{lastweak11est} by 
\begin{align*}
\frac{1}{\lambda^p}\Vert \Omega\Vert_{L^q(\mathbb{S}^{mn-1})} \sum_{j=1}^{m}  \limsup_{l\to \infty} \Big( \prod_{i=1}^{j-1}\Vert f_i^{k_l}\Vert_{L^{p_i}(\bbrn)}^p\Big) \big\Vert f_j-f_j^{k_l}\big\Vert_{L^{p_j}(\bbrn)}^p  \Big( \prod_{i=j+1}^{m}\Vert f_i\Vert_{L^{p_i}(\bbrn)}^p\Big),
\end{align*}
which clearly vanishes.
This completes the proof of Theorem \ref{maincor}.

\hfill



\end{document}